\documentclass[12pt,reqno]{amsart}
 
\usepackage{amsmath,amsthm,amssymb,mathrsfs,amscd,epsfig,color}
\usepackage{url}
\usepackage{graphicx}
\usepackage{mathrsfs}
\usepackage{ulem}
\usepackage{fancybox}
\usepackage{pb-diagram} 
\usepackage[all]{xy}
\usepackage{comment}
\usepackage{mathtools}
\usepackage{centernot}
\usepackage{placeins} 
\usepackage{url}

\makeatletter

\newcommand{\confrac}[2]{%
  \frac{\displaystyle{%
    \strut\hfill{#1}\hfill\;\vrule}}%
      {\displaystyle{%
       \strut\vrule\;\hfill{#2}\hfill}}}%

\makeatletter

    \newcommand\contFrac{\@ifstar{\@contFracStar}{\@contFracNoStar}}

   \def\singleContFrac#1#2{%
        \begin{array}{@{}c@{}}%
            \multicolumn{1}{c|}{#1}%
            \\%
            \hline%
           \multicolumn{1}{|c}{#2}%
        \end{array}%
   }

    \def\@contFracNoStar#1{%
        \mathchoice{
            \@contFracNoStarDisplay@#1//\@nil%
        }{
            \@contFracNoStarInline@#1//\@nil%
        }{
            \@contFracNoStarInline@#1//\@nil%
        }{
            \@contFracNoStarInline@#1//\@nil%
        }%
    }

    \def\@contFracNoStarDisplay@#1//#2\@nil{%
        \@ifmtarg{#2}{%
            #1%
        }{%
            #1+\cfrac{1}{\@contFracNoStarDisplay@#2\@nil}%
        }%
    }

        \def\@contFracNoStarInline@#1//#2\@nil{%
            \@ifmtarg{#2}{%
                #1%
            }{%
                #1 \@@contFracNoStarInline@@#2\@nil%
            }%
        }
        \def\@@contFracNoStarInline@@#1//#2\@nil{%
            \@ifmtarg{#2}{%
                + \singleContFrac{1}{#1}%
            }{%
                + \singleContFrac{1}{#1} \@@contFracNoStarInline@@#2\@nil%
            }%
        }

    \def\@contFracStar#1{%
        \mathchoice{
            \@contFracStarDisplay@#1////\@nil%
        }{
            \@contFracStarInline@#1//\@nil%
        }{
            \@contFracStarInline@#1//\@nil%
        }{
            \@contFracStarInline@#1//\@nil%
        }%
    }

    \def\@contFracStarDisplay@#1//#2//#3\@nil{%
        \@ifmtarg{#2}{%
            #1%
        }{%
            #1 + \cfrac{#2}{\@contFracStarDisplay@#3\@nil}%
        }%
    }

        \def\@contFracStarInline@#1//#2\@nil{%
            \@ifmtarg{#2}{%
                #1%
            }{%
                #1 \@@contFracStarInline@@#2\@nil%
            }%
        }
        \def\@@contFracStarInline@@#1//#2//#3\@nil{%
            \@ifmtarg{#3}{%
                + \singleContFrac{#1}{#2}%
            }{%
                + \singleContFrac{#1}{#2} \@@contFracStarInline@@#3\@nil%
            }%
        }
\makeatother

       \numberwithin{equation}{section}
\theoremstyle{plain}

\newtheorem{thm}{Theorem}[section]
\newtheorem{lem}[thm]{Lemma}
\newtheorem{cor}[thm]{Corollary}

\newtheorem{pro}[thm]{Proposition}
\theoremstyle{definition}

\newtheorem{rem}[thm]{Remark}

\newtheorem*{prf*}{Proof}

\newtheorem*{pf*}{}
\newtheorem*{lem*}{LemmaA}
\newtheorem*{lm*}{LemmaB}
\makeatletter
\@namedef{subjclassname@2020}{%
  \textup{2020} Mathematics Subject Classification}
\makeatother
\title[Density combinatorics theorems in continued fractions]
{Density combinatorics theorems
 in\\ fractal dimension theory of continued fractions }
\author{Yuto Nakajima and Hiroki Takahasi}

\address{Faculty of Science and Engineering, Doshisha University, Kyoto, 610-0394, JAPAN}
\email{yunakaji@mail.doshisha.ac.jp}
\address{Keio Institute of Pure and Applied Sciences (KiPAS),  Department of Mathematics, Keio University, Yokohama, 223-8522, JAPAN}  \email{hiroki@math.keio.ac.jp}

\begin{document}

\begin{abstract} 
We build a bridge from 
density combinatorics 
to fractal dimension theory of continued fractions.
We establish a fractal transference principle that  transfers common properties of subsets of $\mathbb N$ with positive upper 
density to properties of subsets of irrationals in $(0,1)$
for which the set $\{a_n(x)\colon n\in\mathbb N\}$ of partial quotients  induces an injection $n\in\mathbb N\mapsto a_n(x)\in\mathbb N$.
Let $(*)$ be a certain property that holds for any subset of $\mathbb N$ with positive upper 
density.  The principle asserts that for any $S\subset\mathbb N$ with positive upper 
density, there exists a set $E_S\subset(0,1)$ of Hausdorff dimension $1/2$ such that the set   $\bigcup_{n\in\mathbb N}\bigcap_{x\in E_S}\{a_n(x)\}\cap S$ has the same upper density as that of $S$, and thus inherits property $(*)$.
Examples of $(*)$
include the existence of arithmetic progressions of arbitrary lengths and the existence of arbitrary polynomial progressions, known as Szemer\'edi's and Bergelson-Leibman's theorems respectively.
In the same spirit,
we establish a relativized version of the principle applicable to
 the primes, to the primes of the form $y^2+z^2+1$ with $y$, $z\in\mathbb N$, to the sets given by the Piatetski-Shapiro sequences.
\end{abstract}

\maketitle

\section{Introduction}\label{intro}

Problems on the existence of arithmetic progressions  lie at the interface of a number of fields, including 
combinatorics, number theory, ergodic theory and dynamical systems. 
 Van der Waerden’s theorem
 \cite{Wae27}
 states that if the set $\mathbb N$ of positive integers is partitioned into finitely many sets, then one of these sets contains 
 an $\ell$-term arithmetic progression for every $\ell\geq3$.
A significant strengthening of Van der Waerden's theorem was obtained by
  Szemer\'edi \cite{S75} 
  who proved 
that any subset of $\mathbb N$ with positive upper density contains an $\ell$-term arithmetic progression for every $\ell\geq3$.
For a set $S\subset\mathbb N$, define its
 {\it upper density} 
 by \[\overline{d}(S)=\limsup_{N\to \infty}\frac{\#({S\cap[1,N]})}{N},\]
and its
 {\it upper Banach density} 
by
\[d_{\rm B}(S)=\limsup_{N\to \infty}\frac{\sup_{M\in \mathbb N}\#({S\cap[M,M+N]})}{N}.\]
Note that $d_{\rm B}(S)\geq \overline{d}(S)$.
The positive upper density assumption in Szemer\'edi's theorem can actually be weakened to the positivity of upper Banach density.
Furstenberg \cite{Fur77,Fur81} established the multiple recurrence theorem and the correspondence principle, linking the field of density combinatorics to recurrence in ergodic theory.
Using this principle he obtained a new proof of Szemer\'edi’s theorem. This 
development laid a foundation for further extensions of
Szemeredi's theorem to several different directions, see \cite{BL,FK78,GT08,SY19,SP19,TZ}
for example.
The aim of this paper is to provide counterparts of these 
 density combinatorics theorems
in the context of fractal dimension theory of continued fractions.

Each irrational number $x$ in the interval $(0,1)$
   has the unique infinite
 {\it regular continued fraction expansion} 
\begin{equation}\label{RCF}x=\confrac{1 }{a_{1}(x)} + \confrac{1 }{a_{2}(x)}+ \confrac{1 }{a_{3}(x)}  +\cdots,\end{equation}
where each positive integer $a_n(x)$, $n\in\mathbb N$ is called a partial quotient or simply 
a {\it digit} of $x$. 
We analyze the structure of the set of digits 
 as an infinite subset of $\mathbb N$. For this purpose, overlapping digits do not count and we may restrict to the set of points whose digits are all different:
\[E=\{x\in (0,1)\setminus\mathbb Q\colon a_{m}(x)\neq a_{n }(x)  \ \text{for all } m,n\in\mathbb N\text{ with }  m\neq n\}.\]
An investigation of this set was suggested by Erd\H{o}s in the early stage of fractal dimension theory of continued fractions.
Ramharter \cite{R85} proved that 
$E$ is of Hausdorff dimension $1/2$.  Ramharter's result is a refinement of Good's 
landmark theorem \cite[Theorem~1]{G} 
which states that the set $\{x\in(0,1)\setminus\mathbb Q\colon a_n(x)\to\infty\text{ as }n\to\infty\}$ is of Hausdorff dimension $1/2$.

\subsection{Statements of results and their consequences}\label{state-sec}

We now state our two main results (Theorems~\ref{cor-FS} and \ref{cor-FS-R}), along with consequential theorems (Theorems~\ref{cor-N}, \ref{cor-prime}, \ref{cor-P1}, \ref{cor-PS} and \ref{lower-cor}) which are manifestations of the density combinatorics theorems 
\cite{BL,GT08,SY19,SP19,TZ}
in the regular continued fraction.
\subsubsection{First main result}
Let
$\dim_{\rm H}$ denote the Hausdorff dimension on $[0,1]$.

\begin{thm}[fractal transference principle]\label{cor-FS}
Let $S\subset\mathbb N$. 
\begin{itemize}
\item[(a)] If $\overline{d}(S)>0$, then there exists $E_S\subset E$ 
such that \begin{equation}\label{equation1}\dim_{\rm H}E_S=\dim_{\rm H}E=\frac{1}{2}\ \text{ and }\ \overline{d}\left(\bigcup_{n\in\mathbb N}\bigcap_{x\in E_S}\{a_n(x)\}\cap S\right)=\overline{d}(S).\end{equation} 
In particular, 
\begin{equation}\label{equation-0}\dim_{\rm H}\left\{x\in E\colon \overline{d}(\{a_n(x)\colon n\in\mathbb N\}\cap S)=\overline{d}(S)\right\}=\dim_{\rm H}E.\end{equation}
\item[(b)] If 
$d_{\rm B}(S)>0$, then
there exists $F_{S}\subset E$ 
such that \begin{equation}\label{equation1'}\dim_{\rm H}F_{S}=\dim_{\rm H}E\
\text{ and }\
d_{\rm B}\left(\bigcup_{n\in\mathbb N}\bigcap_{x\in F_S}\{a_n(x)\}\cap S\right)=d_{\rm B}(S).\end{equation}
In particular, 
\begin{equation}\label{equation-1}\dim_{\rm H}\left\{x\in E\colon d_{\rm B}(\{a_n(x)\colon n\in\mathbb N\}\cap S)=d_{\rm B}(S)\right\}=\dim_{\rm H}E.\end{equation}
\end{itemize}\end{thm}

The second equation in \eqref{equation1} asserts that the set of elements of $S$ that appear in the sequence of partial quotients of any $x\in E_S$ in the same position is of upper density $\overline{d}(S)$. As a `pointwise' version of this, \eqref{equation-0} asserts that the condition
$\overline{d}(\{a_n(x)\colon n\in\mathbb N\}\cap S)=\overline{d}(S)$ does not cause a dimension drop in $E$. 

Equation \eqref{equation-0} allows us to 
 immediately transfer statements on subsets of $\mathbb N$ with positive upper density to statements on continued fractions of points in sets of Hausdorff dimension $1/2$. For example, 
 from \eqref{equation-0}
and Szemer\'edi's theorem it follows that if  $S\subset\mathbb N$ and $\overline{d}(S)>0$, then for any $x\in(0,1)\setminus\mathbb Q$ in a set of Hausdorff dimension $1/2$ and for every integer $\ell\geq3$, there exist $k(x)\in\mathbb N$, $m(x)\in\mathbb N$,
$n_1(x),n_2(x),\ldots,n_\ell(x)\in\mathbb N$
such that  
\[k(x)+m(x),k(x)+2m(x),\ldots,k(x)+\ell m(x)\in S,\] 
 and \[a_{n_1(x)}(x)=k(x)+m(x), a_{n_2(x)}(x)=k(x)+2m(x), \ldots,a_{n_\ell(x)}(x)=k(x)+\ell m(x).\] 
Equations \eqref{equation1} 
ensure that all the integers $k(x),m(x),n_1(x),n_2(x),\ldots,n_\ell(x)$ 
 are taken  uniformly on the set $E_S$ of Hausdorff dimension $1/2$.
 Moreover, a close inspection into Section~\ref{insert-sec} involving the construction of $E_S$ shows that  one can choose the indices
$n_1(x),n_2(x),\ldots,n_\ell(x)$ so that 
$n_1(x)<n_2(x)<\cdots<n_\ell(x)$.
Furthermore, from the uniformity in multiple recurrence (see \cite[Theorem~F2]{BHMF00}) and Furstenberg's correspondence principle, the common difference $m(x)$ $(x\in E_S)$ of the $\ell$-term arithmetic progression can be taken to be bounded from above by a universal constant that depends only on $\ell$ and $\overline{d}(S)$.
If $\overline{d}(S)=0$ and $d_{\rm B}(S)>0$, 
then similar statements follow from \eqref{equation1'} and \eqref{equation-1}.
 This uniformity of integers and the monotonicity of indices apply to Theorems~\ref{cor-N}, \ref{cor-prime}, \ref{cor-P1}, \ref{cor-PS} and \ref{lower-cor}. 
 \subsubsection{Polynomial progressions in $\mathbb N$}
We proceed to more examples of uniform transfer via Theorem~\ref{cor-FS}.
Bergelson and Leibman established a multiple recurrence theorem with polynomial times \cite[Theorem~${\rm A}_0$]{BL}, and as its corollary derived a density combinatorics theorem 
\cite[Theorem~${\rm B}_0$]{BL} which in particular asserts that any subset of $\mathbb Z$ with positive upper Banach density contains `arbitrary polynomial progressions'.
 Their theorem is a vast generalization of Szemer\'edi's theorem.
 For an application to the continued fraction,
 we need the following version of Bergelson-Leibman's theorem in $\mathbb N$. For completeness we include a proof in Appendix~\ref{BL-appendix}.
Let $\mathbb Z_0[X]$ denote the set of polynomials with integer coefficients that vanish at $X=0$.

\begin{thm}[cf. {\cite[Theorem~${\rm B}_0$]{BL}}]\label{BL-thm} 
If $S\subset\mathbb N$ and $d_{\rm B}(S)>0$, then for any finite collection $h_1,h_2\ldots,h_\ell$ of elements of $\mathbb Z_0[X]$ 
there exist $k\in\mathbb N$, $m\in\mathbb N$ such that 
\[k+h_1(m),k+h_2(m),\ldots,k+h_\ell(m)\in S.\]
\end{thm}

 Transferring Theorem~\ref{BL-thm} to the continued fraction
via Theorem~\ref{cor-FS} we obtain the following statement. 
\begin{thm}\label{cor-N}
If $S\subset\mathbb N$ and
  $\overline{d}(S)>0$ (resp. 
  $d_{\rm B}(S)>0$), then for any finite collection $h_1,h_2,\ldots,h_\ell$ of elements of $\mathbb Z_0[X]$  there exist 
  $k\in\mathbb N$, $m\in\mathbb N$, $n_1,n_2\ldots,n_\ell\in\mathbb N$ such that  \[k+h_1(m),k+h_2(m),\ldots,k+h_\ell(m)\in S,\]
  and 
  for any $x\in E_S$ (resp. $x\in F_S$), \[a_{n_1}(x)=k+h_1(m),a_{n_2}(x)=k+h_2(m),\ldots,a_{n_\ell}(x)=k+h_\ell(m).\]
\end{thm}

\subsubsection{Second main result}
Theorem~\ref{cor-N} does not apply to subsets of $\mathbb N$ whose upper Banach density is zero. 
To accommodate such sets, notably the set of primes,
 we develop a relativized version of Theorem~\ref{cor-FS}.
Given two infinite sets $A\subset S\subset\mathbb N$, {\it the upper density of $A$ relative to $S$} is defined by
\[\overline{d}(A|S)=\limsup_{N\to\infty}\frac{\#( A\cap[1,N])}{\#(S\cap[1,N])}.\]
We say $S$ has {\it polynomial density} if there exist constants
 $C_1>0$, $C_2>0$ and $\alpha\geq1$, $\beta\geq0$ such that
for all sufficiently large $N>1$,
\[\frac{C_1N^{1/\alpha}}{(\log N)^{\beta }}\leq \#(S\cap[1,N])\leq \frac{C_2N^{1/\alpha}}{(\log N)^{\beta }}.\]
In this case we also say $S$ has {\it polynomial density with exponent} $\alpha$.
Our second main result is 
a relativized version of Theorem~\ref{cor-FS}.

\begin{thm}[relative fractal transference principle]\label{cor-FS-R}
Let $S$ be an infinite subset of $\mathbb N$ that has polynomial density with exponent $\alpha\geq1$.
For any $A\subset S$ with $\overline{d}(A|S)>0$, 
there exists a subset $E_{S,A}$ of $\{x\in E \colon \{a_n(x)\colon n\in\mathbb N\}\subset  S \}$ such that \begin{equation}\label{equation4}\dim_{\rm H}E_{S,A}=\dim_{\rm H}\{x\in E\colon \{a_n(x)\colon n\in\mathbb N\}\subset S\}=\frac{1}{2\alpha}\end{equation} and
\begin{equation}\label{equation5}\overline{d}\left(\bigcup_{n\in\mathbb N}\bigcap_{x\in E_{S,A}}\{a_n(x)\}\cap A|S\right)=\overline{d}(A|S).\end{equation}
In particular, 
\begin{equation}\label{equation6}\dim_{\rm H}\left\{
\begin{tabular}{l}
\!\!\!$x\in E\colon$\\
\!\!\!$\{a_n(x)\colon n\in\mathbb N\}\subset S$,\!\!\!\\
\!\!\!$\overline{d}(\{a_n(x)\colon n\in\mathbb N\}\cap A |S)$\!\!\!\\
\!\!\!$=\overline{d}(A|S)$\!\!\!
\end{tabular}
\right\}=\dim_{\rm H}
\left\{\begin{tabular}{l}
\!\!\!$x\in E\colon$\\
\!\!\!$\{a_n(x)\colon n\in\mathbb N\}\subset S$\!\!\!\end{tabular}
\right\}.\end{equation}
\end{thm}

Equation  \eqref{equation5} asserts that the set of elements of $A$ that appear in the sequence of partial quotients of any $x\in E_{S,A}$ in the same position is of relative upper density $\overline{d}(A|S)$. As a `pointwise' version of this,
\eqref{equation6} asserts that the condition
$\overline{d}(\{a_n(x)\colon n\in\mathbb N\}\cap A|S)=\overline{d}(A|S)$ does not cause a dimension drop in $E$.

Since $\mathbb N$ has polynomial density with $\alpha=1$, $\beta=0$, Theorem~\ref{cor-FS-R} includes Theorem~\ref{cor-FS}(a) as a special case. 
Equation \eqref{equation6} is a
 `pointwise' version of the first equality in \eqref{equation4}, and
allows us to
immediately transfer density combinatorics statements on subsets of $S$ to statements in fractal dimension theory of continued fractions. 
This transfer can be done uniformly on $E_{S,A}$ by virtue of \eqref{equation5}, 
as developed below.

\subsubsection{Polynomial progressions in primes}
The first application of Theorem~\ref{cor-FS-R} is to the set $\mathbb P$ of prime numbers. 
It is well-known (see e.g., \cite[Corollary~3.4]{MV07}) that the upper Banach density of $\mathbb P$ is zero.
Green and Tao \cite[Theorem~1.2]{GT08} proved that
any subset of $\mathbb P$ with positive upper relative density contains an $\ell$-term arithmetic progression for every $\ell\geq3$, which they call `Szemer\'edi's theorem in prime numbers'. Green-Tao's theorem is a corollary of the wide open conjecture of Erd\H{o}s and Tur\'an \cite{ET36} which states that any subset of $\mathbb N$  with divergent reciprocal sum contains arithmetic progressions of arbitrary lengths. Subsequently, Tao and Ziegler \cite{TZ} extended Green-Tao's theorem to polynomial progressions as follows.
\begin{thm}[{\cite[Theorem~1.3]{TZ}}]\label{TZ-thm} 
If $A\subset\mathbb P$ and $\overline{d}(A|\mathbb P)>0$, then
for any finite collection $h_1,h_2,\ldots,h_\ell$ of elements of $\mathbb Z_0[X]$ 
there exist $k\in\mathbb Z$, $m\in\mathbb N$ such that
\[k+h_1(m),k+h_2(m),\ldots,k+h_\ell(m)\in A.\]
\end{thm}

From the prime number theorem it follows that $\mathbb P$ has polynomial density with $\alpha=\beta=1$.  
Transferring Theorem~\ref{TZ-thm} to the continued fraction via Theorem~\ref{cor-FS-R}  
we obtain the following statement.

\begin{thm}\label{cor-prime}
If $A\subset\mathbb P$ 
 and $\overline{d}(A|\mathbb P)>0$, then
for any finite collection $h_1,h_2,\ldots,h_\ell$ of elements of $\mathbb Z_0[X]$ there exist 
 $k\in\mathbb Z$, $m\in\mathbb N$, $n_1,n_2,\ldots,n_\ell\in\mathbb N$ such that \[k+h_1(m),k+h_2(m),\ldots,k+h_\ell(m)\in A,\]
 and for any $x\in E_{\mathbb P,A},$ \[a_{n_1}(x)=k+h_1(m),a_{n_2}(x)=k+h_2(m),\ldots,a_{n_\ell}(x)=k+h_\ell(m).\] 

\end{thm}

Taking $A=\mathbb P$ in Theorem~ \ref{cor-prime} yields the following statement.
\begin{cor}
We have
\[\dim_{\rm H}\left\{
\begin{tabular}{l}
\!\!\!$x\in E\colon \{a_n(x)\colon n\in\mathbb N\}$ is contained in $\mathbb P$ and contains \!\!\!\\
\!\!\!\! an $\ell$-term arithmetic progression for every $\ell\geq3$
\!\!\!\!
\end{tabular}
\right\}=\frac{1}{2}.\]
\end{cor}

Studying the set of continued fractions with restrictions involving the set of prime numbers is not new, see \cite{DaSi24,GHWW,MW99,SZ23} for example. 
Mauldin and Urba\'nski \cite{MW99},
Das and Simmons \cite{DaSi24} investigated the set of irrationals in $(0,1)$ all whose regular continued fraction digits are prime numbers. Schindler and Zweim\"uller \cite{SZ23} investigated probabilistic laws governing the appearance of prime regular continued fraction digits for irrationals which are typical in the sense of the Lebesgue measure on $(0,1)$.
Gonz\'alez Robert et al. \cite{GHWW} computed the Lebesgue measure and the Hausdorff dimension of certain sets of irrationals in $(0,1)$ with infinitely many large prime digits.

\subsubsection{Arithmetic progressions in  primes of quadratic form }In this subsection
we give three more examples of application of Theorem~\ref{cor-FS-R}, one of which is still related to prime numbers.
Let $\mathbb P(1)$ denote the set of prime numbers represented by a quadratic form $x^2+y^2+1$, i.e.,
\[\mathbb P(1)=\{p\in\mathbb P\colon  p=x^2+y^2+1 \ \text{\rm for some }
x,y\in\mathbb N\}.\]
Linnik \cite{Lin60} proved that $\mathbb P(1)$ contains infinitely many primes. Motohashi \cite{Mo71} proved an upper bound
$\#(\mathbb P(1)\cap [1,N])\leq CN/(\log N)^{\frac{3}{2}}$, and Iwaniec \cite[Theorem~1]{I72} obtained a lower bound of the same order,
namely
$\mathbb P(1)$ has polynomial density with $\alpha=1$, $\beta=3/2$.
For more information on $\mathbb P(1)$, see \cite{Mat07,Wu98} for example.

Ter\"av\"ainen \cite{Tera18} proved that `almost every' even positive integer $n\not\equiv5,8$ mod $9$ can be written as the sum of two elements of $\mathbb P(1)$. As a by-product of a proof of this, he also proved that any subset of 
the set
$\{p\in\mathbb P\colon  p=x^2+y^2+1 \ \text{\rm for some coprime }
x,y\in\mathbb N\}$ 
with positive relative upper density contains infinitely many non-trivial $3$-term arithmetic progressions.
Sun and Pan \cite{SP19} extended Green-Tao's theorem to  $\mathbb P(1)$. 
\begin{thm}[{\cite[Theorem~1.1]{SP19}}]\label{SP-thm} 

If $A\subset\mathbb P(1)$ 
 and $\overline{d}(A|\mathbb P(1))>0$, 
then $A$ contains an $\ell$-term arithmetic progression for every $\ell\geq3$.
\end{thm}

Transferring Theorem~\ref{SP-thm} to the continued fraction via Theorem~\ref{cor-FS-R}  
we obtain the following statement.
\begin{thm}\label{cor-P1}
If $A\subset\mathbb P(1)$ 
 and $\overline{d}(A|\mathbb P(1))>0$, then 
for any integer $\ell\geq3$ there exist  $k\in\mathbb Z$, $m\in\mathbb N$, $n_1,n_2,\ldots,n_\ell\in\mathbb N$ such that \[k+m,k+2m,\ldots,k+\ell m\in A,\] and
for any $x\in E_{\mathbb P(1),A}$, \[a_{n_1}(x)=k+m,a_{n_2}(x)=k+2m,\ldots,a_{n_\ell}(x)=k+\ell m.\] 

\end{thm}

\subsubsection{Arithmetic progressions in sets given by Piatetski-Shapiro sequences }
For $x\geq0$ let $\lfloor x\rfloor$ denote the largest integer not exceeding $x$.
For $\alpha>1$ we consider the set
\[{\rm PS}(\alpha)=\{\lfloor n^\alpha\rfloor\colon n\in\mathbb N\}.\]
For $\alpha>1$ and $\alpha\notin\mathbb N$, the sequence $(\lfloor n^\alpha\rfloor )_{n=1}^\infty$
is called {\it Piatetski-Shapiro's sequence} \cite{PS53},
 after Piatetski-Shapiro \cite{PS53} who proved that for all $1<\alpha<12/11$ the sequences contain infinitely many prime numbers. Rivat and Wu \cite{RW01} improved this range to $1<\alpha<243/205$. 
In spite of a vast amount of literature,
the existence of prime numbers in Piatetski-Shapiro's sequences is not a topic of this paper.
There is no integer $\alpha>1$ such that ${\rm PS}(\alpha)$ contains arithmetic progressions of arbitrary lengths  \cite{DM97,D52,D66}. 
Since ${\rm PS}(\alpha)$ has a convergent reciprocal sum, it is independent of the conjecture of Erd\H{o}s and Tur\'an.

\begin{figure}
\begin{center}
\includegraphics[height=7cm,width=10.5cm]{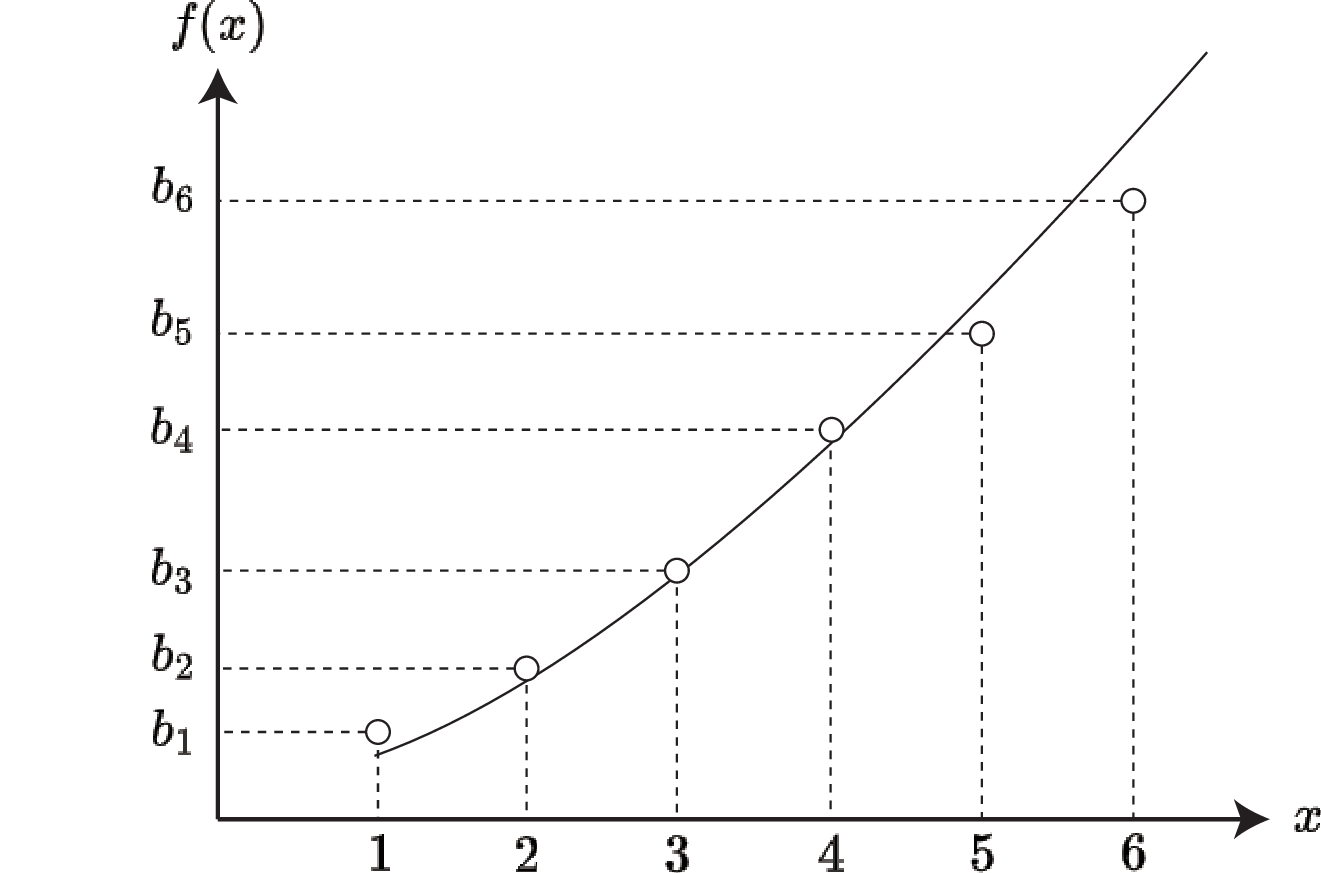}
\caption{The graph of a smooth function $x\in[1,\infty)\mapsto f(x)\in\mathbb R$ and an associated slightly curved sequence $\{b_n\}_{n=1}^\infty$: $|b_n-f(n)|=O(1)$ $(n\to\infty)$.}
\label{fig10}
\end{center}
\end{figure}
 Piatetski-Shapiro's sequence is
obtained by curving with bounded errors 
the sequence $(n^\alpha)_{n=1}^\infty$ whose graph $\{(n,n^\alpha)\in\mathbb R^2\colon n\in\mathbb N\}$ is contained in the graph of the smooth function $x\in[1,\infty)\mapsto x^\alpha$.
Positive integer sequences with this kind of properties are called {\it slightly curved sequences} \cite{SY19},
see \textsc{Figure}~\ref{fig10}.
 Saito and Yoshida \cite[Theorem~4]{SY19} extended Szemer\'edi's theorem to graphs of a class of slightly curved sequences, and as a corollary proved that for every $1<\alpha<2$, the graph $\{(n,\lfloor n^\alpha\rfloor)\in\mathbb N^2\colon n\in\mathbb N\}$ contains an $\ell$-term arithmetic progression for every $\ell\geq3$  (see \cite[Corollary~6]{SY19}). From this it follows that 
 ${\rm PS}(\alpha)$ contains an $\ell$-term arithmetic progression for every $\ell\geq3$.
In fact,
 the next theorem follows from combining results in \cite{SY19}. 
 For completeness we include a proof in Appendix~\ref{pfthmSY}.

\begin{thm}\label{thm-SY}Let $1<\alpha<2$. If $A\subset{\rm PS}(\alpha)$
and $\overline{d}(A|{\rm PS}(\alpha))>0$,
then $A$ contains an $\ell$-term arithmetic progression for every $\ell\geq3$.\end{thm}

Clearly, ${\rm PS}(\alpha)$ has polynomial density with exponent $\alpha$ and $\beta=0$.
Transferring Theorem~\ref{thm-SY} to the continued fraction via Theorem~\ref{cor-FS-R} 
we obtain the following statement.
\begin{thm}\label{cor-PS}
Let $1<\alpha<2$. If $A\subset{\rm PS}(\alpha)$
and $\overline{d}(A|{\rm PS}(\alpha))>0$, then for any integer $\ell\geq3$ there exist $k\in\mathbb Z$, $m\in\mathbb N$, $n_1,n_2,\ldots,n_\ell\in\mathbb N$ such that \[k+m,k+2m,\ldots,k+\ell m\in A,\]
and
for any $x\in E_{{\rm PS}(\alpha),A}$, \[a_{n_1}(x)=k+m,a_{n_2}(x)=k+2m\ldots,a_{n_\ell}(x)=k+\ell m.\]  
\end{thm}

\subsubsection{Polynomial progressions in subsets of $\mathbb N$ with positive lower density}
The lower density of $S\subset\mathbb N$ is given by  \[\underline{d}(S)=\liminf_{N\to \infty}\frac{\#({S\cap[1,N]})}{N}.\]
Any subset of $\mathbb N$ with positive lower density has polynomial density with $\alpha=1$, $\beta=0$. Transferring Theorem~\ref{BL-thm} to the continued fraction via Theorem~\ref{cor-FS-R} we obtain the following statement.
\begin{thm}\label{lower-cor}
Let $S\subset\mathbb N$ satisfy $\underline{d}(S)>0$.
If $A\subset S$ and $\overline{d}(A|S)>0$, then for any finite collection $h_1,h_2\ldots,h_\ell$ of elements of $\mathbb Z_0[X]$ there exist $k\in\mathbb Z$, $m\in\mathbb N$, $n_1,n_2,\ldots,n_\ell\in\mathbb N$ such that \[k+h_1(m),k+h_2(m),\ldots,k+h_\ell(m)\in S,\]
and for any $x\in E_{S,A}$, \[a_{n_1}(x)=k+h_1(m),a_{n_2}(x)=k+h_2(m),\ldots,a_{n_\ell}(x)=k+h_\ell(m).\] \end{thm}

\subsection{Essential role of non-monotonicity of digits}\label{non}
Arithmetic progressions in the continued fraction was first studied by Tong and Wang \cite{TW}, who considered the set 
\[J=\{x\in (0,1)\setminus\mathbb Q\colon a_n(x)<a_{n+1}(x)\ \text{ for every } n\geq1\}\subsetneq E,\] 
and proved that the set
of $x\in J$ whose digit set $\{a_n(x)\colon n\in\mathbb N\}$ contains an $\ell$-term arithmetic progression with a common difference $m$
for every $\ell\geq3$ and every $m\geq1$
is of Hausdorff dimension $1/2$. 
The argument in \cite{TW} is constructive and does not rely on Szemer\'edi's theorem.
Zhang and Cao \cite{ZC2} proved that 
the set of points in $J$ whose digit sets have upper Banach density $1$ is of Hausdorff dimension $1/2$, a statement comparable to \eqref{equation-1} in the special case $S=\mathbb N$.

 Since the map $x\in J\mapsto \{a_n(x)\colon n\in\mathbb N\}\in 2^{\mathbb N}$ induces a bijection between $J$ and the collection of infinite subsets of $\mathbb N$, 
 $J$ is indeed a natural set to investigate. However, the next dimension drop theorem of independent interest 
 suggests that $E$ rather than $J$ is a right choice for our purpose of building a bridge from density combinatorics to fractal dimension theory of continued fractions.
\begin{thm}\label{J-thm}
Let $S\subset\mathbb N$. 
\begin{itemize}
\item[(a)] 
We have \[\dim_{\rm H}\left\{x\in J\colon \overline{d}(\{a_n(x)\colon n\in\mathbb N\}\cap S)>0\right\}=0<\dim_{\rm H}J.\]
\item[(b)]
If $S$ has polynomial density with exponent $\alpha<2$, then for any $A\subset S$ we have
\[\dim_{\rm H}\left\{
\begin{tabular}{l}
\!\!\!$x\in J\colon$\\
\!\!\!$\{a_n(x)\colon n\in\mathbb N\}\subset S$,\!\!\!\\
\!\!\!$\overline{d}(\{a_n(x)\colon n\in\mathbb N\}\cap A |S)>0$\!\!\!
\end{tabular}
\right\}\leq\frac{\alpha-1}{2\alpha}<\dim_{\rm H}
\left\{\begin{tabular}{l}
\!\!\!$x\in J\colon$\\
\!\!\!$\{a_n(x)\colon n\in\mathbb N\}\subset S$\!\!\!\end{tabular}
\right\}.\]
\end{itemize}\end{thm}

In a sharp contrast to \eqref{equation-0},
Theorem~\ref{J-thm}(a) asserts that 
 for any $S\subset\mathbb N$ with positive upper density, the condition $\overline{d}(\{a_n(x)\colon n\in\mathbb N\}\cap S)=\overline{d}(S)$ forces a significant dimension drop in $J$. In contrast to 
\eqref{equation6},
 Theorem~\ref{J-thm}(b) asserts that
the conditions
$\{a_n(x)\colon n\in\mathbb N\}\subset S$ and 
$\overline{d}(\{a_n(x)\colon n\in\mathbb N\}\cap A |S)=\overline{d}(A|S)$ together force a dimension drop in $J$. The Hausdorff dimension drops to zero in the case $\alpha=1$ such as $S=\mathbb P$ and $S=\mathbb P(1)$.
 These comparisons show that
the non-monotonicity of continued fraction digits for points in $E\setminus J$ is crucial
in both Theorems~\ref{cor-FS} and \ref{cor-FS-R}.
For more technical details on the role of non-monotonicity of digits, see the comments in (III) in Section~\ref{outlinepf}.

\subsection{Outline of proofs of the main results}\label{outlinepf}
Since $\overline{d}(S|\mathbb N)=\overline{d}(S)$ holds for any $S\subset \mathbb N$,
Theorem~\ref{cor-FS}(a) follows from Theorem~\ref{cor-FS-R}.
 Proofs of Theorems~\ref{cor-FS}(b) and \ref{cor-FS-R} rely on essentially the same set of ideas,
 but there are some differences. Below we outline proofs of these two.

In principle, constructions of all our target sets in $E$ break into three steps.
\begin{itemize}
\item {\bf Step~1:} (Section~\ref{pre-lem-sec}) construction of an infinite subset of $\mathbb N$.

\item {\bf Step~2:} (Section~\ref{PF-sec}) construction of a {\it seed set}: a certain Moran fractal in $E$ consisting of points whose digit sets are contained in the set constructed in Step~1.

\item {\bf Step~3:} (Sections~\ref{insert-sec}, \ref{prevent-sec}) modification of
the seed set constructed in Step~2 without changing the Hausdorff dimension. 
\end{itemize}
Step~3 is most technical.
A Moran fractal is a Cantor-like set 
that is constructed by iterative procedure 
based on countably many parameters, and
has been used extensively for estimating Hausdorff dimension of interesting sets, see  \cite{MW88,Mo46,P97,PW97} for example.

 Let $S\subset\mathbb N$
have polynomial density with exponent $\alpha\geq1$, and let $A\subset S$ have positive upper density relative to $S$.
The three steps in
the construction of the target set $E_{S,A}$ in Theorem~\ref{cor-FS-R} are described as follows.

\begin{itemize}
\item {\bf Step~1:} construction of 
an infinite subset $S_*$ of $S$ with zero upper density relative to $S$.

\item {\bf Step~2:} construction of a Moran fractal $R_{t,L}(S_*)$ (see \eqref{seed-def}) of Hausdorff dimension $1/(2\alpha)$  consisting of points whose digit sets are contained in $S_*$.
This set is our seed set.

\item {\bf Step~3:} 
modification of the seed set $R_{t,L}(S_*)$ into a new 
Moran fractal of the same Hausdorff dimension, 
by
adding elements of 
 $A$ into the digit sequences of points in the seed set. This new Moran fractal is our target set $E_{S,A}$ (see \eqref{target-def}).
\end{itemize}

\noindent  

This `seed set $\rightsquigarrow$ target set construction' has been 
inspired by the afore-mentioned papers 
 \cite{TW,ZC2} in which the monotonicity of digits is crucial.  
However, as pointed out in Section~\ref{non}, for proofs of our main results
 it is essential to treat points in $E$ for which the digit sequences are not monotone.
 To this end,
 we have developed
a collection of new ideas summarized as follows.

\begin{itemize}
\item[(I)] Since the goal is to construct a subset of $E$, overlapping digits must be avoided. This means that, in the construction of $E_{S,A}$, 
elements of $A$ 
already appearing in the digit sequences of points in the seed set cannot be added in Step~3. In other words, 
what we actually do in Step~3 
is to add elements of $A\setminus S_*$ rather than $A$.
According to
\eqref{equation6}, the elements of $A$ must appear in the digit set of any point in $E_{S,A}$ with relative upper density $\overline{d}(A|S)$.
Hence, the upper density of $S_*$ relative to $S$ must be zero.
Due to this restriction, the construction of $S_*$ in Step~1 is elaborated.

\item[(II)] 

By the well-known formula in \cite[pp.71--72]{Fal14}, the Hausdorff dimension of a Moran fractal in $[0,1]$ is estimated from below
in terms of their parameters. 
The seed set we construct in Step~2 and modify in Step~3
consists of points whose digit sets are contained in the set $S_*$ constructed in Step~1. Due to this restriction of digits, 
detecting parameters of our seed set becomes delicate. We do this parameter detection in a unified way, incorporating the case where $S$ has polynomial density with exponent $\alpha\geq1$. 
As a by-product, our construction yields a refinement of  Ramharter’s theorem \cite{R85} of independent interest (see Proposition~\ref{E-cor}).

\item[(III)]  Intuitively, adding more digits in Step~3 may result in a more loss of Hausdorff dimension. The task in Step~3 of the construction of $E_{S,A}$ is to add sufficiently many elements of $A\setminus S_*$ in appropriate positions, to achieve the positive density 
$\overline{d}(A|S)$ without losing Hausdorff dimension at all.
We accomplish 
this  
counter-intuitive task 
by providing a sharp estimate of H\"older exponent of the associated {\it elimination map}. This estimate follows from our careful choices of positions of digits to be added, which  
relies on the non-monotonicity of continued fraction digits for points in $E\setminus J$. 
 See Section~\ref{insert-sec} for a general setup of Step~3 and relevant definitions, and 
 Section~\ref{prevent-sec} for the key estimate on the elimination map.

\end{itemize}

By virtue of Theorem~\ref{cor-FS}(a), in order to prove Theorem~\ref{cor-FS}(b) 
we may start with a set $S\subset\mathbb N$ with $\overline{d}(S)=0$ and $d_{\rm B}(S)>0$ (see Section~\ref{pf-sec}).
The three steps in the construction of the target set $F_S$
in Theorem~\ref{cor-FS}(b)
 are described as follows.

\begin{itemize}
\item {\bf Step~1:} construction of 
an infinite subset $(\mathbb N\setminus S)_*$ of $\mathbb N\setminus S$ with zero upper density relative to $\mathbb N\setminus S$.

\item {\bf Step~2:} construction of a Moran fractal $R_{t,L}((\mathbb N\setminus S)_*)$ of Hausdorff dimension $1/2$  consisting of points whose digit sets are contained in $(\mathbb N\setminus S)_*$.
This set is our seed set.

\item {\bf Step~3:} 
modification of the seed set $R_{t,L}((\mathbb N\setminus S)_*)$ into a new 
Moran fractal of the same Hausdorff dimension, 
by
adding elements of 
 $S$ into the digit sequences of points in the seed set. This new Moran fractal is our target set $F_{S}$ (see \eqref{target-def-F}).
\end{itemize}
\noindent Due to the difference between the upper density and the upper Banach density, 
 Step~3 of the construction of $F_S$ is considerably different from Step~3 of the construction of $E_{S,A}$.
  We have designed 
 Step~3 in Sections~\ref{insert-sec} and  \ref{prevent-sec} with a sufficient generality, 
 to accommodate both constructions in a unified manner.

 \subsection{Structure of the paper}
The rest of this paper consists of four sections and two appendices.
 In Section~2 we summarize preliminary results.  
 In Section~3 we carry out the construction of a Moran fractal in the three steps 
 outlined in Section~\ref{outlinepf}.
 In Section~4 we apply the general construction in Section~3, and complete the proofs of Theorems~\ref{cor-FS} and \ref{cor-FS-R}. We also complete a proof of Theorem~\ref{J-thm}.
In Section~5, we generalize our main results to a wide class of Iterated Function Systems and semi-regular continued fractions.
In  Appendices~A and B we prove Theorems~\ref{BL-thm} and \ref{thm-SY} respectively.

\section{Preliminaries}
This section summarizes preliminary results needed for the proofs of Theorems~\ref{cor-FS}, \ref{cor-FS-R} and \ref{J-thm}. In Section~\ref{CF-sec} we
recall basic properties of continued fraction digits in the expansion \eqref{RCF}. 
In Section~\ref{dimension-sec} 
we introduce a Moran fractal, 
and provide two lower bounds of Hausdorff dimension of sets in $[0,1]$.
 In Section~\ref{conv-sec} we 
introduce a convergence exponent for infinite subsets of $\mathbb N$, and investigate its properties.

\subsection{Fundamental intervals for continued fractions}\label{CF-sec}
For each $n\in \mathbb N$
and $(a_1,\ldots, a_n)\in \mathbb N^n$,
define non-negative integers
 $p_n$ and  $q_n$ by the recursion formulas
\begin{equation}\label{pq}\begin{split}p_{-1}=1,\ p_0=0,\ p_i&=a_ip_{i-1}+p_{i-2} \ \text{ for  } i=1,\ldots,n,\\
q_{-1}=0,\ q_0=1,\ q_i&=a_iq_{i-1}+q_{i-2} \ \text{ for  } i=1,\ldots,n.\end{split}\end{equation}
For $n\in\mathbb N$ and
$(a_1,\ldots, a_n)\in \mathbb N^n$, we define an {\it $n$-th fundamental interval} by
\[ I(a_1,\ldots, a_n)=
\begin{cases}\vspace{1mm}
\left[\displaystyle{\frac{p_n}{q_n}}, \frac{p_n+p_{n-1}}{q_n+q_{n-1}}\right)&\text{if}\ n\ \text{is even,}\\
\displaystyle{\left(\frac{p_n+p_{n-1}}{q_n+q_{n-1}}, \frac{p_n}{q_n}\right]}&\text{if}\ n\ \text{is odd.}
 \end{cases}
 \]
This interval represents the set of elements of $(0, 1]$ which have the finite or infinite regular continued fraction expansion beginning by $a_1,\ldots,a_n$, i.e.,
\[I(a_1,\ldots, a_n)=\{x\in (0, 1]\colon a_i(x)=a_i\ \text{ for }i=1,\ldots,n\}.\]
Let $|\cdot|$ denote the Euclidean diameter of sets in $[0,1]$. 
 The next lemma estimates the diameter of each fundamental interval in terms of its string of digits.
\begin{lem}
\label{proper}
For every $n\in \mathbb N$ and every $(a_1,\ldots, a_n)\in \mathbb N^n$ we have
\[\frac{1}{2}\prod_{i=1}^{n}\frac{1}{(a_{i}+1)^{2}}\le |I(a_1,\ldots, a_n)|\le\prod_{i=1}^{n}\frac{1}{a_{i}^{2}}.\]
\end{lem}
\begin{proof}
By \cite{IK, Khi64}, for every $n\in \mathbb N$ and every $(a_1,\ldots, a_n)\in \mathbb N^n$ we have
\begin{equation}\label{pq-lem}|I(a_1,\ldots, a_n)|=\frac{1}{q_n(q_n+q_{n-1})}\ \text{ and }\ \prod_{i=1}^n a_i\le q_n\le \prod_{i=1}^n(a_i+1),\end{equation}
where $q_{n-1}$, $q_n$ are given by \eqref{pq}. 
Since $q_{n-1}<q_n$ we have 
\[\frac{1}{2q_n^2}\leq|I(a_1,\ldots, a_n)|\leq  \frac{1}{q_n^2}.\]
Combining these with the double inequalities in \eqref{pq-lem} yields the desired ones.
\end{proof}

 \subsection{Lower bounds of Hausdorff dimension}\label{dimension-sec}
Let $F=\bigcap_{n=0}^\infty E_n$ be a set in $[0,1]$ such that the following holds for every $n\geq1$: 
\begin{itemize}
\item $E_{n-1}$ is the union of finitely many closed intervals in $[0,1]$ of positive diameter, called {\it $(n-1)$-th level intervals};
\item for each $(n-1)$-th level interval $I$, 
the number of $n$-th intervals contained in $E_n\cap I$ is at least $r_n$, and all of these $n$-th intervals are separated by gaps of diameter at least $\delta_n$, where $r_n\geq2$ and $0<\delta_{n+1}<\delta_n$.
\end{itemize}
 We call $F$ {\it a Moran fractal with parameters} $(r_n)_{n=1}^\infty$, $(\delta_n)_{n=1}^\infty$. The Hausdorff dimension of a Moral fractal is estimated from below with its parameters by the next lemma.
\begin{lem}[{\cite[pp.71--72]{Fal14}}]
\label{mass}
Let $F\subset [0,1]$ be a Moran fractal with parameters
 $(r_n)_{n=1}^\infty$, $(\delta_n)_{n=1}^\infty$.
Then 
\[\dim_{\rm H}F\ge \liminf_{n\to \infty}\frac{\log(r_1r_2\cdots r_{n-1})}{-\log (r_n\delta_n)}.\]
\end{lem}

Let $F\subset [0,1]$ be a set. We say $f\colon F\to [0,1]$ is {\it H\"older continuous with exponent} $\gamma\in(0,1]$ if there exists $C>0$ such that
\[|f(x)-f(y)|\le C|x-y|^{\gamma}\ \text{ for all } x, y\in F.\]
\begin{lem}[{\cite[Proposition~3.3]{Fal14}}]
\label{Holder-F}
Let $F\subset [0,1]$ and let $f\colon F\to [0,1]$ be H\"older continuous with exponent $\gamma\in(0,1]$.
Then \[\dim_{\rm H}F\geq\gamma\cdot\dim_{\rm H}f(F).\]\end{lem}

We say $f\colon F\to [0,1]$ is {\it almost Lipschitz} if
for any $\gamma\in(0,1)$,
$f$ is H\"older continuous with exponent $\gamma$.
If $f$ is not Lipschitz continuous and almost Lipschitz, then the multiplicative constant $C$ grows to infinity as $\gamma$ increases to $1$, but this growth is
 inconsequential in the 
 lower bound of Hausdorff dimension.
Increasing $\gamma$ to $1$ in Lemma~\ref{Holder-F} we obtain the next lemma.
\begin{lem}\label{Holder}Let $F\subset [0,1]$ and let $f\colon F\to [0,1]$ be almost Lipschitz. Then \[\dim_{\rm H} F\geq\dim_{\rm H}f(F).\]
\end{lem}

\subsection{Convergence exponent}\label{conv-sec}
 For an infinite subset $S$ of $\mathbb N$, define
\[\tau(S)=\inf\left\{t\geq0\colon\sum_{ a\in S} a^{-t}<\infty\right\}.\]
We call $\tau(S)$ a {\it convergence exponent} of $S$.
 Note 
 that $\tau(S)\in[0,1]$.
 If we write $S=\{a_n\colon n\in\mathbb N\}$, $a_1<a_2<\cdots$ then
\begin{equation}\label{formula-tau}\tau(S)=\limsup_{n\to \infty}\frac{\log n}{\log a_n},\end{equation}
see \cite[Section~2]{PS72}. The convergence exponent is often used in fractal dimension theory, see \cite{Hir73,WW08} for example. 
The next lemma can be proved by a standard covering argument by fundamental intervals.
 \begin{lem}[{\cite[Corollary~1]{Hir73}}]\label{upper} For any infinite subset $S$ of $\mathbb N$, we have \[\dim_{\rm H}\{x\in(0,1)\setminus\mathbb Q\colon \{a_n(x)\colon n\in\mathbb N\}\subset S\text{ and }a_n(x)\to\infty\text{ as }n\to\infty\}\leq\frac{\tau(S)}{2}.\]\end{lem}\begin{rem}The inequality in Lemma~\ref{upper} is actually the equality for any infinite subset $S$ of $\mathbb N$, known as a solution of Hirst's conjecture \cite{WW08}. The inequality suffices for the proofs of our main results. The equality for $S=\mathbb N$ is known as Good's theorem \cite[Theorem~1]{G}. \end{rem}

The next lemma links the convergence exponent to relative upper density.
We will mostly use it with $A=S$ in combination with Lemma~\ref{upper}, in order to obtain an upper bound of Hausdorff dimension.
 The case $A\subsetneq S$ will only be used in the proof of Theorem~\ref{J-thm}.

\begin{lem}\label{disc2}If $S\subset\mathbb N$ has polynomial density with exponent $\alpha\geq1$, then for any $A\subset S$ with $\overline{d}(A|S)>0$ we have
\[\tau(A)=\frac{1}{\alpha}.\]
\end{lem}
\begin{proof}
Write $A=\{a_n\colon n\in\mathbb N\}$, $a_1<a_2<\cdots$. Since $S$ has polynomial density with exponent $\alpha$, there exist $C_2>0$ and $\beta\geq0$ such that for all sufficiently large $n$ we have
\[\frac{\log n}{\log a_{n}}=\frac{\log \#(A\cap [1, a_{n}])}{\log a_{n}}\le \frac{\log (C_2/(\log a_{n})^{\beta})}{\log a_{n}}+\frac{1}{\alpha}.\]
Letting $n\to\infty$ and combining the result with \eqref{formula-tau} we obtain $\tau(A)\le 1/\alpha.$

In the case $A=S$, the same argument as above yields the reverse inequality $\tau(A)\ge 1/\alpha.$ In the case $A\subsetneq S$,
 the same argument does not work 
since $\overline{d}(A|S)>0$ does not imply $\limsup_{n\to\infty}\#(A\cap[1,a_{n}])/\#(S\cap[1,a_n])>0$. 
Take a strictly increasing sequence $(N_k)_{k=1}^\infty$ of positive integers such that \[\lim_{k\to\infty}\frac{\#(A\cap [1,N_k])}{\#(S\cap [1,N_k])}=\overline{d}(A|S)>0.\]  Let $\varepsilon\in(0,1/\alpha)$. Since $S$ has polynomial density with exponent $\alpha\geq1$, there exist  $C_1>0$ and $\beta\geq0$ such that for all sufficiently large $k$ we have \[\begin{split}\frac{\#(A\cap [N_k^\varepsilon,N_k])}{\#(S\cap [1,N_k])}&\geq\frac{\#(A\cap [1,N_k])}{\#(S\cap [1,N_k])}-\frac{\#(A\cap [1,N_k^\varepsilon])}{\#(S\cap [1,N_k])}\\&\geq\frac{\overline{d}(A|S)}{2}-\frac{(\log N_k)^{\beta}}{C_1N_k^{1/\alpha-\varepsilon }}\geq\frac{\overline{d}(A|S)}{3}.\end{split}\] Rearranging this gives \[\#(A\cap [N_k^\varepsilon,N_k])\geq\frac{\overline{d}(A|S)}{3}\#(S\cap [1,N_k])\geq\frac{\overline{d}(A|S)}{3}\frac{C_1N_k^{1/\alpha}}{(\log N_k)^\beta}.\] Let $0<t<1/\alpha$. For all sufficiently large $k\geq1$ we have \[ \sum_{a\in A\cap[N_k^\varepsilon,N_k]} a^{-t} >N_k^{-t}\#(A\cap [N_k^\varepsilon,N_k])\geq \frac{\overline{d}(A|S)}{3}\frac{C_1N_k^{1/\alpha-t}}{(\log N_k)^\beta},\] 
Letting $k\to\infty$ we conclude that  $\sum_{a\in A} a^{-t} =\infty$, namely $\tau(A)\geq t$. Increasing $t$ to $1/\alpha$ we obtain $\tau(A)\geq1/\alpha$.\end{proof}

\begin{rem}\label{referee-x}
The difference between the upper density and the upper Banach density appear in their relations with convergence exponent.
By Lemma~\ref{disc2}, if $A\subset\mathbb N$ and $\overline{d}(A)>0$ then $\tau(A)=1$.
It is easy to see that $d_{\rm B}(A)>0$ does not imply $\tau(A)=1$:
for example, for the set $A=\{\left(\bigcup_{n\in\mathbb N}[n^2,n^2+\sqrt{n}]\right)\cap\mathbb N\}$ we have $d_{\rm B}(A)=1$ and $\tau(A)=3/4$. 
\end{rem}

The next result on the multifractal analysis of convergence exponent for continued fractions \cite{FMSW21} is a key ingredient in the proof of Theorem~\ref{J-thm}.
\begin{lem}[{\cite[(2.2) on p.1899]{FMSW21}}]\label{fang-eq}
For all $0\leq c\leq 1$
we have
\[\dim_{\rm H}\{x\in J\colon\tau(\{a_n(x)\colon n\in\mathbb N\})=c\}=\frac{1-c}{2}.\]
\end{lem}
\begin{rem}The difference between a fractal structure of $J$ and that of $E$ appear in their relations with the multifractal spectrum of convergence exponent. Taking $S=\mathbb N$ in \eqref{equation-0} and combining the result with Lemma~\ref{disc2}, we get
\[\dim_{\rm H}\{x\in E\colon\tau(\{a_n(x)\colon n\in\mathbb N\})=1\}=\frac{1}{2},\]
which is
a sharp contrast to the equality in Lemma~\ref{fang-eq} with $c=1$.\end{rem}

\section{Three-step general construction}
 In this section we perform the 
three-step general construction of Moran fractals in $E$ described in Section~\ref{outlinepf}. 
In Section~\ref{C-const} we introduce  {\it extreme seed sets} that we are going to construct in Step~2. 
In Section~\ref{pre-lem-sec} 
we complete Step~1, and
in Section~\ref{PF-sec} 
complete Step~2.
The remaining two subsections are on Step~3 that is most technical.
In Section~\ref{insert-sec},
we describe a procedure of constructing a new Moran fractal in $E$ by adding digits into a given seed set. 
In Section~\ref{prevent-sec} we prove that an elimination map associated with the construction in Section~\ref{insert-sec} is almost Lipschitz.
\subsection{Extreme seed sets}\label{C-const}
Let $S$ be an infinite subset of $\mathbb N$.
For an infinite subset $K$ of $S$ and positive reals $t,L$ satisfying $t>L>1$, define
\begin{equation}\label{seed-def}R_{t,L}(K)=\{x\in E\colon a_n(x)\in K\cap [t^n,Lt^{n}]
\ \text{ for every}\ n \ge 1\}.\end{equation}
Note that $R_{t,L}(K)\subset J$.
Non-empty sets of this form are called {\it seed sets}. 
 We say the seed set $R_{t,L}(K)$ is an {\it extreme seed set associated with $S$} if the following conditions hold:
 \begin{itemize}
 \item[(A1)] $S$ has polynomial density with exponent $\alpha\geq1$;
 \item[(A2)] $K$ does not contain consecutive integers; \item[(A3)]
 $\overline{d}(K|S)=0;$ 
 \item[(A4)] $\dim_{\rm H}R_{t,L}(K)=1/(2\alpha)$.
 \end{itemize}

 By Lemmas~\ref{upper} and \ref{disc2}, 
 if (A1) holds then  the Hausdorff dimension of any seed set associated with $S$ does not exceed $1/(2\alpha)$. 
Hence,
extreme seed sets are of least relative density of digits (A3)
and of
maximal Hausdorff dimension (A4). Condition (A1) 
is not  essential, but 
a little facilitates 
the proof of Proposition~\ref{holder-ex-lem1} in Section~\ref{insert-sec}.

\begin{pro}
\label{seed-Prop}
 If $S\subset\mathbb N$ has polynomial density, then there exist $S_*\subset S$ and positive reals $t,L$
such that $t>L>1$, $t>2$ and 
$R_{t,L}(S_*)$ is an extreme seed set associated with $S$.
\end{pro}

Prior to a proof of Proposition~\ref{seed-Prop}, 
as Step~1
 we construct in Section~\ref{pre-lem-sec} a relatively thin  subset of a given infinite subset of $\mathbb N$.
The proof of Proposition~\ref{seed-Prop} is given in Section~\ref{PF-sec} as Step~2.

\subsection{Construction of a relatively thin subset (Step~1)}\label{pre-lem-sec}
The next general lemma allows us to extract
from an arbitrary infinite subset of $\mathbb N$
a set of zero relative density with definite growth bounds.
Define a function $\nu\colon [1,\infty)\to\mathbb N$ by \begin{equation}\label{mu-def}\nu(\xi)=k\ \text{ for } \xi\in[k!,(k+1)!),\ k\in\mathbb N.\end{equation}

\begin{lem}\label{find-C}Let $S$ be an infinite subset of $\mathbb N$ and write $S=\{a_n\colon n\in\mathbb N\}$, $a_1<a_2<\cdots$. There exists $S_*\subset S$ that does not contain consecutive integers and satisfies \begin{equation}\label{seed-eq1}\frac{1}{2\nu(\xi)}\leq \frac{\#\{1\leq n\leq \xi\colon a_n\in S_* \}}{\xi}\leq \frac{3}{\nu(\xi)}\end{equation} for all sufficiently large $\xi\geq1$. In particular, $\overline{d}(S_*|S)=0$.
\end{lem}

For a proof of Lemma~\ref{find-C}, we first show that $\nu$ 
is a slow increasing function with the following upper bound.
\begin{lem}\label{simple}For any $t>e$ there exists $\lambda>1$ such that \[\nu(t^n)\leq (n\log t)^{\lambda}\ \text{ for every }n\in\mathbb N.\]\end{lem}
\begin{proof}
We claim that $\nu(t^n)\leq n\log t$ for all $n\in\mathbb N$ but finitely many ones.
For otherwise the definition of $\nu$ in \eqref{mu-def} would yield $t^n>\lfloor n\log t\rfloor!$ for infinitely many $n$, which is a contradiction. 
There exists a constant $c(t)\geq1$ such that
$\nu(t^n)\leq c(t)n\log t$ for every $n\in\mathbb N$. Take $\lambda>1$
such that $c(t)\leq(\log t)^{\lambda-1}$. \end{proof}

\begin{proof}[Proof of Lemma~\ref{find-C}]
For each $k\in\mathbb N$, taking every $k$ elements of $\mathbb N$ with respect to the natural order one can construct a subset of $\mathbb N$ with density $1/(k+1)$. Our construction here relies on this simple observation.
Let \[Q=\bigcup_{k=2}^\infty\{k!+ik\colon i=0,\ldots, k!-1\}\subset\mathbb N.\]
 The following interpretation of $Q$ should be useful along the way. Split $[2,\infty)=\bigcup_{k=2}^\infty[k!,(k+1)!)$. For each integer $k\geq2$, further split $[k!,(k+1)!)$ into consecutive left-closed, right-open subintervals containing exactly $k$ integers:
\[[k!,(k+1)!)=[k!,k!+k)\cup[k!+k,k!+2k)\cup\cdots\cup[k!+ (k!-1)k,(k+1)!).\]
Then $Q$ consists of the left endpoints of all these subintervals. 
Note that $Q$ does not contain consecutive integers.
Set \[S_*=\{a_n\in S\colon n\in Q\}.\]
Then $S_*$ does not contain consecutive integers.

For every $\xi\geq1$ we have
\[\#\{1\leq n\leq \xi\colon a_n\in S_* \}=\#(Q\cap[1,\xi]).\]
If $\xi\geq1$ is sufficiently large then we have
\[\begin{split}\#(Q\cap[1,\xi])&= \#(Q\cap[1,(k-1)!))+\#(Q\cap[(k-1)!,k!))+\#(Q\cap[k!,\xi ])\\
&\leq \#(Q\cap[1,(k-1)!))+\frac{\#([(k-1)!,k!)\cap\mathbb N)}{k-1}+\frac{\#([k!,\xi ]\cap\mathbb N)}{k}\\
&\leq (k-1)!+(k-1)!+\frac{\#([k!,\xi]\cap\mathbb N)}{k}
\leq\frac{3\xi}{\nu(\xi)},
\end{split}\]
and 
\[\begin{split}\#(Q\cap[1,\xi])&\geq \sum_{i=2}^{k-1}\frac{\#([i!,(i+1)!)\cap\mathbb N)}{i}+\frac{\#( [k!,\xi ]\cap\mathbb N)}{k}-1\\
&\geq \sum_{i=2}^{k-1}\frac{\#([i!,(i+1)!)\cap\mathbb N)}{k}+\frac{\#( [k!,\xi ]\cap\mathbb N)}{k}-1\\
&=\frac{\#([2,\xi]\cap\mathbb N)}{k}-1\geq\frac{\xi}{2\nu(\xi)},\end{split}\]
where $k\geq1$ is the integer with $\xi\in[k!,(k+1)!)$.
This verifies \eqref{seed-eq1}.

The ordering of elements of $S$ as in the statement of Lemma~\ref{find-C} implies 
that for any $N\geq1$ with 
$S\cap[1,N]\neq\emptyset$, we have 
\[\#(S_*\cap[1,N])=\#\{1\leq n\leq
\#(S\cap[1,N])\colon a_n\in S_* \}.\]
By the second inequality in \eqref{seed-eq1}, for all sufficiently large $N\geq1$ we have
\[\frac{\#(S_*\cap[1,N])}{\#(S\cap[1,N])}=\frac{\#\{1\leq n\leq
\#(S\cap[1,N])\colon a_n\in S_* \}}{\#(S\cap[1,N]) }
\leq\frac{3}{\nu(\#(S\cap[1,N]) )}.\]
Letting $N\to\infty$ we obtain
$\overline{d}(S_*|S)=0$ as required.
\end{proof}

\subsection{Proof of Proposition~\ref{seed-Prop} (Step~2)}\label{PF-sec} 
Let $S\subset\mathbb N$ have polynomial density with exponent $\alpha\geq1$ and write $S=\{a_n\colon n\in\mathbb N\}$, $a_1<a_2<\cdots.$  
Let $S_*$ be a subset of $S$ for which the conclusion of Lemma~\ref{find-C} holds.
It suffices to show that $\dim_{\rm H}R_{t,L}(S_*)\geq1/(2\alpha)$ for some $t,L$ with $t>L>1$ and $t>2$.

By the ordering of elements of $S$ as above, for any $N\geq1$ with $S\cap [1,N]\neq\emptyset$ we have
$\#(S_*\cap [1,N])
=\#\{1\leq n\leq \#(S\cap [1,N])\colon a_n\in S_*\}$.
Moreover \eqref{seed-eq1} gives
\begin{equation}\begin{split}\label{seed-eq8}\frac{\#(S\cap [1,N])}{
2\nu(\#(S\cap [1,N])) }\leq\#(S_*\cap [1,N])
&\leq \frac{3\#(S\cap [1,N])}{\nu(\#(S\cap [1,N]))}.\end{split}\end{equation}

Let $C_1>0$, $C_2>0$, $\beta\geq0$ be the constants in the definition of polynomial density of $S$ with exponent $\alpha$.
Set 
\begin{equation}\label{L}
L=(12C_1^{-1}C_22^\beta)^\alpha>1.\end{equation}
Let $t>\max\{L,e\}$ satisfy
 \begin{equation}\label{seed-eq4}\nu(\#(S\cap[1,t]))\geq2.\end{equation}
 Since $t>L$, if $t$ is sufficiently large then for every $n\in\mathbb N$ we have
\begin{equation}\label{seed-eq0}C_1C_2^{-1}2^{-\beta}L^{-1/\alpha}\leq\frac{\#(S\cap [1,t^n])}{\#(S\cap [1,Lt^n])}\leq C_1^{-1}C_22^{\beta}L^{-1/\alpha}.\end{equation}
For each $n\in\mathbb N$, fix $k=k(n)\in\mathbb N$ such that
$k!\leq\#(S\cap [1,t^n])<(k+1)!.$
If $n$ is sufficiently large so that
$k+2\geq (C_1C_2^{-1}2^{-\beta}L^{-1/\alpha})^{-1}$, then
\[k!\leq\#(S\cap [1,t^n])\leq\#(S\cap [1,Lt^n])\leq\frac{\#(S\cap [1,t^n])}{C_1C_2^{-1}2^{-\beta}L^{-1/\alpha}}<(k+2)!.\]
By the definition of $\nu$ in \eqref{mu-def}, this estimate implies
 \begin{equation}\label{seed-eq2}0\leq\nu(\#(S\cap [1,Lt^n]))-\nu(\#(S\cap [1,t^n]))\leq 1.\end{equation}
By $t>e$ and Lemma~\ref{simple}, there exists $\lambda>1$ such that
for every $n\in\mathbb N$, 
\begin{equation}\label{seed-eq5}\nu(\#(S\cap[1,t^n]))\leq\nu(t^n)\leq C_1(n\log t)^{\lambda }.\end{equation}

 Using \eqref{seed-eq8}, \eqref{seed-eq4} and then \eqref{seed-eq0}, \eqref{seed-eq2}, 
 for every $n\in\mathbb N$ we have
 \[\begin{split}\#(S_*\cap [t^{n}, Lt^{n}])&\geq \#(S_*\cap [1,Lt^n])-\#(S_*\cap [1,t^n])\\
&\geq\frac{\#(S\cap[1,Lt^n])}{ 2\nu(\#(S\cap[1,Lt^n]))}-\frac{3\#(S\cap[1,t^n])}{\nu(\#(S\cap[1,t^n]))}\\
&\geq\frac{C_1C_2^{-1} 2^{-\beta}L^{1/\alpha}\#(S\cap[1,t^n])}{ 2(\nu(\#(S\cap[1,t^n])+1) }-\frac{3\#(S\cap[1,t^n])}{\nu(\#(S\cap[1,t^n]))}.\end{split}\]
By \eqref{seed-eq4}, the last expression is bounded from below by
\[\#(S\cap[1,t^n])
\left(\frac{C_1C_2^{-1}2^{-\beta}
L^{1/\alpha} }{3\nu(\#(S\cap[1,t^n])) }-\frac{3}{\nu(\#(S\cap[1,t^n]))}\right),\]
which equals $\#(S\cap[1,t^n])/\nu(\#(S\cap[1,t^n]))$ by \eqref{L}. Set $\rho=\beta+\lambda$.
Using \eqref{seed-eq5} and the assumption that $S$ has polynomial density, we obtain 
\begin{equation}\label{low-estimate}\begin{split}\#(S_*\cap [t^{n}, Lt^{n}])\geq\frac{\#(S\cap[1,t^n])}{\nu(\#(S\cap[1,t^n]))}&\geq C_1\frac{t^{\frac{n}{\alpha}}}{(n\log t)^{\beta}}\cdot\frac{1}{C_1(n\log t)^{\lambda }}\\
&=\frac{t^{\frac{n}{\alpha}}}{(n\log t)^{\rho }}.\end{split}\end{equation}

 We now set
\[F=\bigcap_{n= 1}^\infty\bigcup_{\substack{ a_i\in {S_*\cap [t^{i}, Lt^{i}]}\\ 1\leq i\leq n}}\overline{I(a_{1},\ldots, a_n)},\]
where the bar denotes the closure.
Since $t>L$, the intervals $[t^{n}, Lt^{n}]$, $n\in\mathbb N$ are pairwise disjoint.
The Moran fractal $F$ is contained in $R_{t,L}(S_*)$
up to countable sets consisting of the boundary points of the level intervals.
For every $n\geq2$ and every $(n-1)$-th level interval of $F$, the number of $n$-th level intervals contained in that interval is at least $ {\lfloor t^{\frac{n}{\alpha}}/(n\log t)^{\rho}\rfloor}$ by \eqref{low-estimate}.
Since $S_*$ does not contain consecutive integers, these $n$-th level intervals are separated at least by some $n$-th fundamental intervals. Hence, the length of each
gap between the $n$-th level intervals is at least
\[\min\{|I(a_{1},\ldots, a_{n})|
\colon t^{i}\leq a_i\le Lt^{i} \ \text{ for } 1\le i\le n\}.\]
By Lemma~\ref{proper}, $t>L$ and $t>2$, we have
\[|I(a_1,\ldots, a_n)|\ge \frac{1}{2}\prod_{i=1}^{n}\frac{1}{(a_{i}+1)^{2}}\ge  \frac{1}{2}\prod_{i=1}^{n}\frac{1}{(Lt^{i}+1)^{2}}\ge  \frac{1}{2}\prod_{i=1}^{n}\frac{1}{(Lt^{i+1})^{2}}.\]
So, $F$ is a Moran fractal with parameters
$({\lfloor t^{\frac{n}{\alpha}}/(n\log t)^{\rho }\rfloor} )_{n=1}^\infty$, 
$((1/2)\prod_{i=1}^{n}(Lt^{i+1})^{-2})_{n=1}^\infty$.  
Lemma~\ref{mass} gives
\[\dim_{\rm H} F\ge
\liminf_{n\to \infty}\frac{\log(2 \prod_{i=1}^{n-1} {\lfloor t^{\frac{i}{\alpha}}/(i\log t)^{\rho }\rfloor} )}{-\log \left((\lfloor t^{\frac{n}{\alpha}}\rfloor/(n\log t)^{\rho })\prod_{i=1}^{n}(Lt^{i+1})^{-2}\right)}.\]

For two eventually positive functions $f$, $g$ on $\mathbb N$, let us use the expression
$f(n)\sim g(n)$ to indicate that $f(n)/g(n)\to1$ as $n\to\infty$.
 Direct calculations show that
\[\begin{split}
\log\left(2 \prod_{i=1}^{n-1} {\left\lfloor\frac{ t^{\frac{i}{\alpha}}}{(i\log t)^{\rho}}\right\rfloor}\right)&=\log 2+\sum_{i=1}^{n-1}\log \left\lfloor\frac{ t^{\frac{i}{\alpha}}}{(i\log t)^{\rho }}\right\rfloor\\
&\sim\log 2+\sum_{i=1}^{n-1}\left(\log t^{\frac{i}{\alpha}}-\log (i\log t)^{\rho }\right)\sim \frac{1}{2\alpha}n^2\log t,
\end{split}\]
and 
\[\begin{split}-\log \left(\frac{\lfloor t^{\frac{n}{\alpha}}\rfloor}{(n\log t)^{\rho }}\prod_{i=1}^{n}\frac{1}{(Lt^{i+1})^{2}}\right)&{\sim}-\frac{n}{\alpha}\log t+\rho \log (n \log t)+\sum_{i=1}^n 2\log (Lt^{i+1})\\
&\sim n^2 \log t.\end{split}\]
 Combining these two expressions yields $\dim_{\rm H}R_{t,L}(S_*)\geq\dim_{\rm H}F\ge1/(2\alpha)$. This completes the proof of Proposition~\ref{seed-Prop}. \qed

 \medskip

In \cite{R85}, Ramharter proved $\dim_{\rm H} J=\dim_{\rm H}E=1/2$.
Combining Proposition~\ref{seed-Prop} with statements on convergence exponent in Section~\ref{conv-sec}, we obtain a refinement of Ramharter’s result.
\begin{pro}\label{E-cor}
If $S\subset\mathbb N$ has polynomial density with exponent $\alpha\geq1$, then 
\[\begin{split}\dim_{\rm H}\{x\in J\colon \{a_n(x)\colon n\in\mathbb N\}\subset S\}&=\dim_{\rm H}\{x\in E\colon \{a_n(x)\colon n\in\mathbb N\}\subset S\}\\&=\frac{1}{2\alpha}.\end{split}\]\end{pro}
\begin{proof}On the one hand, by Lemmas~\ref{upper} and \ref{disc2} we have
\[\dim_{\rm H}\{x\in E\colon \{a_n(x)\colon n\in\mathbb N\}\subset S\}\leq\frac{1}{2\alpha}.\]
On the other hand, 
Proposition~\ref{seed-Prop} implies \[\dim_{\rm H}\{x\in J\colon \{a_n(x)\colon n\in\mathbb N\}\subset S\}\geq\frac{1}{2\alpha}.\] 
Since $J\subset E$, these two inequalities yield the desired equalities. \end{proof}

\subsection{Adding digits into seed sets (Step~3)}\label{insert-sec}
Let $A$, $B$ be infinite subsets of $\mathbb N$.
Let $t>L>1$ and suppose $R_{t,L}(B)\neq\emptyset$. 
We add elements of $A$ into the digit sequences of points in the seed set $R_{t,L}(B)$, and construct a new subset of $(0,1)\setminus\mathbb Q$
in the following manner.

Let $M_0=0$, and let $(M_k)_{k=1}^\infty$ be a strictly increasing sequence of positive integers. Let $(W_k)_{k=1}^\infty$ be a sequence of finite subsets of $A$. 
Let $y\in R_{t,L}(B)$ and put $a_0(y)=0$ for convenience.
For each $k\geq1$ with $W_k\neq\emptyset$, we
 add the elements of $W_{k}$ into the digit sequence $(a_n(y))_{n=1}^\infty$ in the increasing order to define a new sequence 
 \[\begin{split}&\ldots,a_{M_{k}-1}(y),\ a_{M_{k}}(y),\ w_{1}^{(k)},\ldots,w_{i(k)}^{(k)},\ a_{M_{k}+1}(y),\ldots\\
&\ldots,a_{M_{k+1 }}(y),\ w_{1}^{(k+1)},\ldots,w_{i(k+1)}^{(k+1)},\ a_{M_{k+1 }+1}(y),\ldots\end{split}\]
 where \[W_k=\{w^{(k)}_i\colon i=1,\ldots,i(k)\},\ w^{(k)}_1<\cdots <w^{(k)}_{i(k)}.\]
Let $x(y)$ denote the point in $(0,1)\setminus\mathbb Q$ whose regular continued fraction expansion \eqref{RCF} is given by this new sequence. Let $R_{t,L}(B,A,(M_k)_{k=1}^\infty,(W_k)_{k=1}^\infty)$ denote the collection of these points:
\[R_{t,L}(B,A,(M_k)_{k=1}^\infty,(W_k)_{k=1}^\infty)=\{x(y)\in (0,1)\setminus\mathbb Q\colon y\in R_{t,L}(B)\}.\] 
The map $y\in R_{t,L}(B)\mapsto x(y)\in R_{t,L}(B,A,(M_k)_{k=1}^\infty,(W_k)_{k=1}^\infty)$ is bijective.
Let \[f_{t,L,B,A,(M_k)_{k=1}^\infty,(W_k)_{k=1}^\infty}\colon  R_{t,L}(B,A,(M_k)_{k=1}^\infty,(W_k)_{k=1}^\infty)\to R_{t,L}(B)\] denote the inverse of this map. We call $f_{t,L,B,A,(M_k)_{k=1}^\infty,(W_k)_{k=1}^\infty}$ an {\it elimination map}. 
Clearly, if $A\cap B=\emptyset$ and the elements of $(W_k)_{k=1}^\infty$ are pairwise disjoint then $R_{t,L}(B,A,(M_k)_{k=1}^\infty,(W_k)_{k=1}^\infty)$ is contained in $E$. Moreover, the elimination map $f$ sends $x\in R_{t,L}(B,A,(M_k)_{k=1}^\infty,(W_k)_{k=1}^\infty)$ to $f(x)\in R_{t,L}(B)$ whose digit sequence is given by  eliminating elements of $\bigcup_{k=1}^\infty W_k$ from the digit sequence of $x$. 

\subsection{Almost Lipschitzness of the elimination map (Step~3)}\label{prevent-sec}
We need to deduce a lower bound of 
the Hausdorff dimension of the set $R_{t,L}(B,A,(M_k)_{k=1}^\infty,(W_k)_{k=1}^\infty)$ constructed in Section~\ref{insert-sec}. This can be done by evaluating the H\"older exponent of the elimination map.
This map is merely H\"older in general and as a result the Hausdorff dimension of  $R_{t,L}(B,A,(M_k)_{k=1}^\infty,(W_k)_{k=1}^\infty)$ may drop from that of $R_{t,L}(B)$. However, under appropriate choices of $(M_k)_{k=1}^\infty$ and $(W_k)_{k=1}^\infty$, the H\"older exponent can be made closer and closer to $1$ on deeper levels of $R_{t,L}(B)$, so that the elimination map becomes almost Lipschitz, as stated in the  proposition below.

\begin{pro}\label{holder-ex-lem1}
Let $(\varepsilon_k)_{k=1}^\infty$ be a decreasing sequence of positive reals converging to $0$.
Let $A$, $B$ be infinite subsets of $\mathbb N$
and suppose $B$ does not contain consecutive integers.
Let $t>L>1$, $t>2$ and suppose $R_{t,L}(B)\neq\emptyset$. Let $M_0=0$ and let $(M_k)_{k=1}^\infty$ be a strictly increasing sequence of positive integers, and
let $(W_k)_{k=1}^\infty$ be a sequence of finite subsets of 
$A$ such that for every $k\geq1$, 
\begin{equation}\label{holder-gen-eq0}\prod_{i\in W_{k}}\frac{1}{(i+1)^{2}}\geq\left(\prod_{i=M_{k-1}+1}^{M_k}\frac{1}{t^{2i}} \right)^{3\varepsilon_k}\ \text{ if }W_k\neq\emptyset\end{equation}
and
\begin{equation}\label{holder-ineq0} \frac{1}{2}\prod_{i=1}^{n+1}\frac{1}{(t^{i+2})^2}\cdot\left(\prod_{i=1}^{n}\frac{1}{t^{2i}}\right)^{3\varepsilon_k}\geq\left(\prod_{i=1}^{n}\frac{1}{t^{2i}}\right)^{1+4\varepsilon_k}\ \text{ for every }n\geq M_k.\end{equation}
Then the elimination map $f=f_{t,L,B,A,(M_k)_{k=1}^\infty,(W_k)_{k=1}^\infty}$ is almost Lipschitz.
 \end{pro}

Regarding a proof of this proposition,
 a key observation is  that the influence of adding the elements of $W_k$ is absorbed into the contribution from the digits $a_{M_{k-1}+1}(y),\ldots,a_{M_k}(y)$,  $y\in R_{t,L}(B)$ by virtue of  \eqref{holder-gen-eq0}.
The role of \eqref{holder-ineq0} is auxiliary, to be used after \eqref{holder-gen-eq0} is.

\begin{proof}[Proof of Proposition~\ref{holder-ex-lem1}]

Let us abbreviate $R_{t,L}(B)$,
 $R_{t,L}(B,A,(M_k)_{k=1}^\infty,(W_k)_{k=1}^\infty)$
 to  $R(B)$,
 $R(B,A)$ respectively.
Let $k\in\mathbb N$.
For a pair $y_1,y_2$ of distinct points in $R(B)$,
let $s(y_1,y_2)$ denote the minimal  integer $n\geq0$ such that $a_{n+1}(y_1)\neq a_{n+1}(y_2).$
Take an integer $N\geq M_{k}$ such that for any pair
$y_1,y_2$ of distinct points in $R(B)$ with $s(y_1,y_2)<M_{k}$,
there exists $i\in\{1,\ldots,N\}$ such that
$a_i(f^{-1}(y_1))\neq a_i(f^{-1}(y_2))$.
Since $\sup_{x\in R(B,A)}\max\{a_1(x),\ldots,a_{N}(x)\}$ is finite, 
there is a constant $C>0$ depending on $N$ such that for points
$x_1, x_2$ in $R(B,A)$ satisfying $a_i(x_1)\neq a_i(x_2)$
for some $i\leq N$, we have 
\begin{equation}\label{Holder-initial}|f(x_1)-f(x_2)|\leq 1\leq C|x_1-x_2|.\end{equation}

Let $x_1, x_2\in R(B,A)$ be distinct and put
 $y_1=f(x_1)$, $y_2=f(x_2)$. 
If $s(y_1,y_2)<M_{k}$ then
\eqref{Holder-initial}
gives $|y_1-y_2|\leq C|x_1-x_2|$.
Suppose $s(y_1,y_2)=n\geq M_{k}$. 
Without loss of generality we may assume $a_{n+1}(y_1)<a_{n+1}(y_2).$ 
For each $q\in\mathbb N$, set
 $m_q=\#W_q$.
 Let $q\geq k$ be the integer such that $M_q\le n< M_{q+1}.$ The definition of the elimination map $f$ implies
\[a_i(x_1)=a_i(x_2)\ \text{ for } i=1,\ldots,n+m_ {1}+\cdots+m_{q},\]
and  
\[a_{n+m_1+\cdots+m_q+1}(x_1)=a_{n+1}(y_1)<a_{n+1}(y_2)=a_{n+m_1+\cdots+m_q+1}(x_2).\]
Since $x_1$ and $x_2$ 
are contained in 
$I(a_1(x_1),\ldots,a_{n+m_1+\cdots+m_q}(x_1) )$ and 
$B$ does not contain consecutive integers, we have
$a_{n+1}(y_2)-a_{n+1}(y_1)\geq2$.
Then $x_1$ and $x_2$ are separated by the fundamental interval 
\[I(a_1(x_1),\ldots, a_{n+m_1+\cdots+m_q}(x_1), a_{n+m_1+\cdots+m_q+1}(x_1)+1).\]

We estimate $|x_1-x_2|$ from below and $|y_1-y_2|$ from above. Using the first inequality in Lemma~\ref{proper} and
$a_i(y_1)\leq Lt^i$ for $1\leq i\leq  n+1$, 
we have
\[\begin{split}|x_1-x_2|&\ge |I(a_1(x_1),\ldots,a_{n+m_1+\cdots+m_q}(x_1), a_{n+m_1+\cdots+m_q+1}(x_1)+1)|\\
&\geq\frac{1}{2\left((a_1(x_1)+1)\cdots (a_{n+m_1+\cdots+m_q}(x_1)+1)(a_{n+m_1+\cdots+m_q+1}(x_1)+1)\right)^{2}}\\
&= \frac{1}{2}\prod_{i=1}^{n+1}\frac{1}{(a_i(y_1)+1)^2}\cdot\prod_{j=1}^q\prod_{i\in W_j}\frac{1}{(i+1)^2}\\
&\geq \frac{1}{2} \prod_{i=1}^{n+1}\frac{1}{(Lt^i+1)^{2}}\cdot\prod_{j=1}^q\left(\prod_{i=M_{j-1}+1}^{M_j}\frac{1}{t^{2i}}\right)^{3\varepsilon_j}\\
&= \frac{ D_k}{2} \prod_{i=1}^{n+1}\frac{1}{(Lt^i+1)^{2}}\cdot
\prod_{j=k+1}^q\left(\prod_{i=M_{j-1}+1}^{M_j}\frac{1}{t^{2i}}\right)^{3\varepsilon_j},\end{split}\]
where 
\[D_k=\prod_{j=1}^{k}\left(\prod_{i=M_{j-1}+1}^{M_j}\frac{1}{t^{2i}}\right)^{3\varepsilon_j}.\]
 To deduce the last inequality we have used  \eqref{holder-gen-eq0}.
Using $t>L$ and $t>2$ 
further to bound the last expression, we obtain 
\begin{equation}\label{holder-ineq3}
\begin{split}|x_1-x_2|&\geq \frac{D_k}{2} \prod_{i=1}^{n+1}\frac{1}{(Lt^i+1)^{2}}
\cdot \prod_{j=1}^q\left(\prod_{i=M_{j-1}+1}^{M_j}\frac{1}{t^{2i}}\right)^{3\varepsilon_{k}}\\
&\geq \frac{D_k}{2} \prod_{i=1}^{n+1}\frac{1}{(t^{i+2})^2}
\cdot\left(\prod_{i=1}^{n}\frac{1}{t^{2i}}\right)^{3\varepsilon_{k}}.\end{split}\end{equation}
Since $s(y_1,y_2)=n\geq M_{k}$, 
\eqref{holder-ineq0} and 
\eqref{holder-ineq3} yield
\begin{equation}\label{holder-ineq3.5}|x_1-x_2|\geq D_k\left(\prod_{i=1}^n\frac{1}{t^{2i}} \right)^{1+4\varepsilon_{k}}.
\end{equation}
The second inequality in Lemma~\ref{proper} gives
\begin{equation}\label{holder-ineq4}\prod_{i=1}^n\frac{1}{t^{2i}} \geq 
\prod_{i=1}^n\frac{1}{a_i(y_1)^{2}}
\geq |I(a_1(y_1),\ldots,a_n(y_1))|\geq |y_1-y_2|.\end{equation}
Combining \eqref{holder-ineq3.5} and \eqref{holder-ineq4} we obtain
\[|x_1-x_2|\ge D_k|y_1-y_2|^{1+4\varepsilon_{k}}.\]
Setting $C_{k}=D_k^{\frac{1}{1+4\varepsilon_{k}}}$ yields 
 $|f(x_1)-f(x_2)|\leq C_k|x_1-x_2|^{\frac{1}{1+4\varepsilon_k}}$. 
Since $\varepsilon_k\to0$ as $k\to\infty$, $f$ is almost Lipschitz. The proof of Proposition~\ref{holder-ex-lem1} is complete.\end{proof}

\section{Proofs of the main results }\label{sec-proof}
In this section we complete the proofs of our main results, 
following the three steps developed in Section~3.
 In Section~\ref{gen-sec}
 we prove Theorem~\ref{cor-FS-R}, which implies
 Theorem~\ref{cor-FS}(a).
 We prove Theorem~\ref{cor-FS}(b) in Section~\ref{pf-sec}. Central to these proofs is to 
 verify the key condition \eqref{holder-gen-eq0}.
 To this end, we carefully choose positions of digits to be added in Step~3, taking advantage of the non-monotonicity of digits for points in $E\setminus J$. Finally in Section~\ref{pfthmj}
  we prove Theorem~\ref{J-thm}.

\subsection{Proof of Theorem~\ref{cor-FS-R}}\label{gen-sec}
Let $S\subset \mathbb N$ have polynomial density with exponent $\alpha\geq1$. By Proposition~\ref{seed-Prop},
there exist $S_*\subset S$ not containing consecutive integers
and positive reals $t$, $L$ with $t>L>1$ and $t>2$
such that $R_{t,L}(S_*)$ is an extreme seed set of Hausdorff dimension $1/(2\alpha)$.
Let $A\subset S$
satisfy $\overline{d}(A|S)>0$. We have
\begin{equation}\label{dim-equality}\overline{d}((A\setminus S_*)|S)=\overline{d}(A|S).\end{equation}

Let $(\varepsilon_k)_{k=1}^\infty$ 
 be a decreasing sequence of positive reals converging to $0$. 
We set $M_0=0$ for convenience, and choose a strictly increasing sequence $(M_k)_{k=1}^\infty$ of positive integers greater than one such that all the following conditions hold:
\begin{equation}\label{up-eq1*}\lim_{k\to\infty}\frac{\#((A\setminus S_*)\cap[1,\sqrt{M_k}])}{\#(S\cap[1,\sqrt{ M_k}])}=\overline{d}(A\setminus S_*|S);\end{equation}
\begin{equation}\label{MK-newnew}\lim_{k\to\infty}\frac{t^{M_{k-1}}}{\sqrt{M_k}^{1/\alpha}}= 0;\end{equation}
\begin{equation}\label{up-eq2*}
\ Lt^{M_{k-1}}+ \sqrt{M_k}<t^{M_k}\ \text{ for every }k\geq1;\end{equation}
\begin{equation}\label{up-eq3*} \frac{1}{(Lt^{M_{k-1}}+\sqrt{M_k}+1)^{2\sqrt{M_k}}}
 \geq \frac{1}{t^{2\varepsilon_k M_k^2}}\geq\left(\prod_{i=M_{k-1}+1}^{M_k}\frac{1}{t^{2i}} \right)^{3\varepsilon_k}\ \text{ for every }k\geq1;\end{equation}
 \begin{equation}\label{up-eq4*}\inf_{n\geq M_k} \prod_{i=1}^{n+1}\frac{1}{(t^{i+2})^2}\cdot\left(\prod_{i=1}^{n }\frac{1}{t^{2i}}\right)^{-1-\varepsilon_k}\geq2 \ \text{ for every }k\geq1.\end{equation}
Clearly \eqref{MK-newnew}, \eqref{up-eq2*}, \eqref{up-eq3*}  hold if $M_k$ is chosen to be sufficiently large compared to 
$M_{k-1}$. The role of
  \eqref{up-eq4*} is auxiliary, which holds if $M_k$ is sufficiently large depending on $\varepsilon_k$. 

For each $k\in\mathbb N$
we put 
\[I_k=[Lt^{M_{k-1}}+1,Lt^{M_{k-1}}+\sqrt{M_k}]\]
and
 \[W_k=(A\setminus S_*)\cap I_k.\]
 Taking a subsequence of $(M_k)_{k=1}^\infty$ if necessary we may assume
\begin{equation}\label{I-limit}\lim_{k\to\infty}\frac{\min I_k}{\max I_k}=0.\end{equation}
Since $S$ has polynomial density, 
it intersects both $I_k$ and $[1,\sqrt{M_k}]$ for all sufficiently large $k$.
In the next two paragraphs we show that the two differences
\[\frac{\#(A\cap I_k)}{\#(S\cap I_k)}-\frac{\#W_k}{\#(S\cap I_k)}=\frac{\#(S_*\cap I_k)}{\#(S\cap I_k)}\]
 and 
  \[\begin{split}
\left|\frac{\#W_k}{\#(S\cap I_k)}-\frac{\#((A\setminus S_*)\cap[1,\sqrt{ M_k}])}{\#(S\cap[1,\sqrt{M_k}])}\right|\leq
&\left|\frac{\#W_k-\#((A\setminus S_*)\cap[1,\sqrt{ M_k}])}{\#(S\cap I_k)}\right|\\
&+\left|\frac{\#(S\cap [1,\sqrt{M_k}])-\#(S\cap I_k)}{\#(S\cap I_k)}\right|
\end{split}\]
converge to $0$ as $k\to\infty$.

Since $S$ has polynomial density with exponent $\alpha\geq1$, there exist constants $C_1>0$ and $\beta\geq0$ such that 
\[\begin{split}&\frac{1}{\#(S\cap I_k)}\leq  \left(\frac{C_1(Lt^{M_{k-1}}+\sqrt{M_k})^{1/\alpha}}{(\log(Lt^{M_{k-1}}+\sqrt{M_k}))^\beta}-Lt^{M_{k-1}}\right)^{-1}.\end{split}\]
Since $S$ has polynomial density with exponent $\alpha\geq1$ and $\overline{d}(S_*|S)=0$, 
there exist $C_2>0$ and for any $\varepsilon>0$ there exists $k_0\geq1$ such that for every $k\geq k_0$,
\[\#(S_*\cap I_k)\leq\#(S_*\cap[1,Lt^{M_{k-1}}+\sqrt{M_k}])\leq\varepsilon\cdot\frac{C_2(Lt^{M_{k-1}}+\sqrt{M_k})^{1/\alpha}}{(\log(Lt^{M_{k-1}}+\sqrt{M_k}))^\beta}.\]
Combining these two inequalities and then using 
 \eqref{MK-newnew}, we obtain
\begin{equation}\label{limzero1}\lim_{k\to\infty}\left|\frac{\#(A\cap I_k)}{\#(S\cap I_k)}-\frac{\#W_k}{\#(S\cap I_k)}\right|=0,\end{equation}
as required.

\begin{figure}
\begin{center}
\includegraphics[height=8cm,width=12cm]{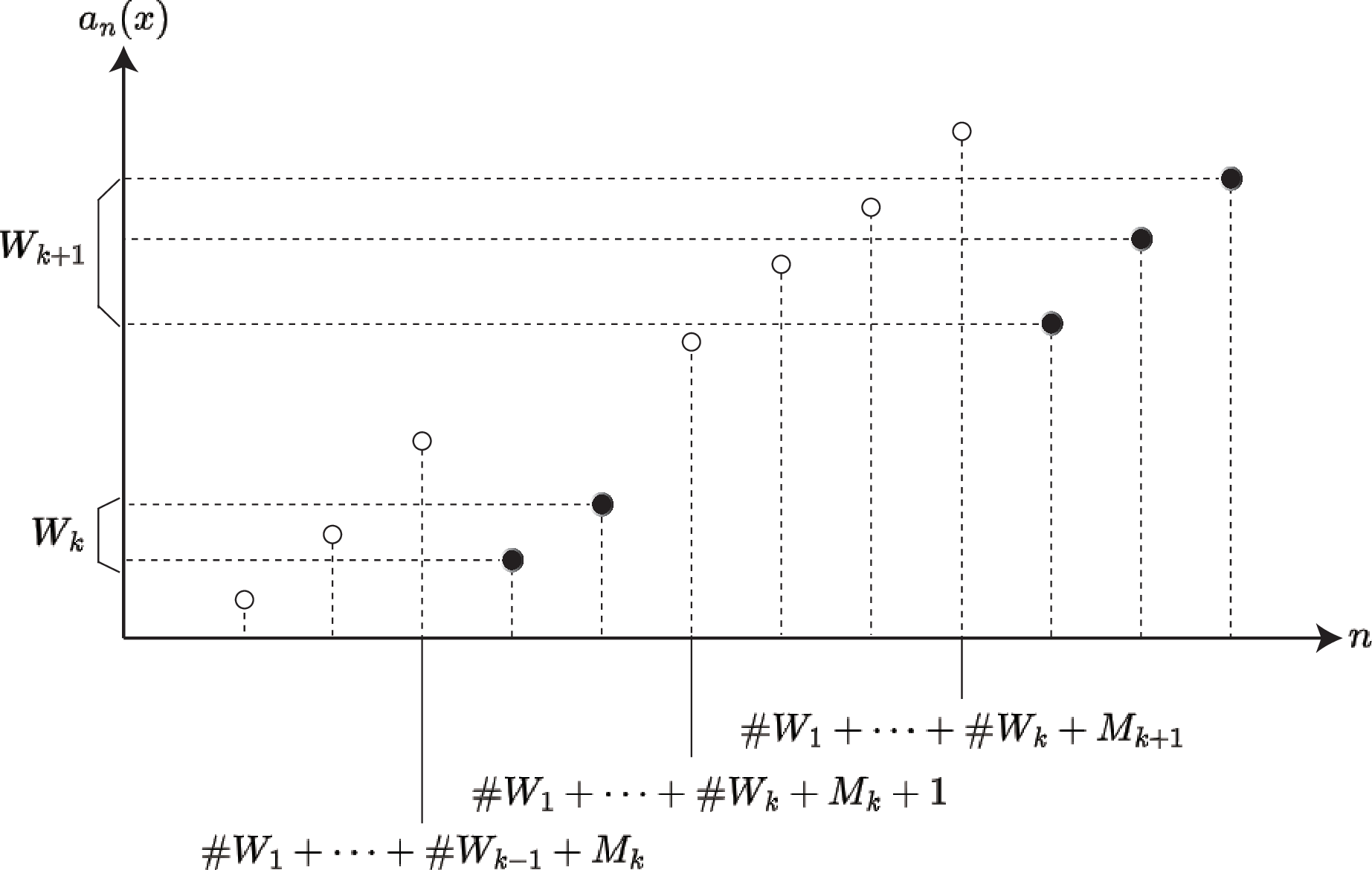}
\caption
{The white and black dots altogether indicate part of the graph of the map $n\in\mathbb N\mapsto a_n(x)\in\mathbb N$ for each $x\in E_{S,A}=R_{t,L}(S_*,A\setminus S_*,(M_k)_{k=1}^\infty,(W_k)_{k=1}^\infty)$. The black dots indicate the digits in $W_k$ (resp. $W_{k+1}$) inserted in between time $M_k$ and time $M_k+1$ (resp. between time $M_{k+1}$ and time $M_{k+1}+1$). }\label{fig1}
\end{center}
\end{figure}

Moving on to the second difference, for the numerator of the first fraction in the right-hand side we have
\[\begin{split}&|\#W_k-\#((A\setminus S_*)\cap[1,\sqrt{ M_k}])|\\
=&|\#((A\setminus S_*)\cap[Lt^{M_{k-1}}+1,\sqrt{ M_k}])+\#((A\setminus S_*)\cap[\sqrt{ M_k}+1,\sqrt{ M_k}+Lt^{M_{k-1}}])\\
&-\#((A\setminus S_*)\cap[1,Lt^{M_{k-1}}])-\#((A\setminus S_*)\cap[Lt^{M_{k-1}}+1,\sqrt{ M_k}])|\leq  Lt^{M_{k-1}}.\end{split}\]
For the numerator of the second fraction, a similar calculation shows
\[|\#(S\cap [1,\sqrt{M_k}])-\#(S\cap I_k)|\leq Lt^{M_{k-1}}.\]
Condition \eqref{MK-newnew} 
implies
\[\lim_{k\to\infty}\frac{\#W_k-\#((A\setminus S_*)\cap[1,\sqrt{ M_k}])}{\#(S\cap I_k)}=0\]
and
\[\lim_{k\to\infty}\frac{\#(S\cap[1,\sqrt{M_k}])-\#(S\cap I_k)}{\#(S\cap I_k)}=0.\]
Hence we obtain
\begin{equation}\label{limzero2}
\lim_{k\to\infty}\left|\frac{\#W_k}{\#(S\cap I_k)}-\frac{\#((A\setminus S_*)\cap[1,\sqrt{M_k}])}{\#(S\cap[1, \sqrt{M_k}])}\right|=0,\end{equation}
as required.

Combining \eqref{limzero1}, \eqref{limzero2} and then using  \eqref{dim-equality}, \eqref{up-eq1*},
\begin{equation}\label{reladens-PS*}
\begin{split}\lim_{k\to\infty}\frac{\#(A\cap I_k)}{\#(S\cap I_k)}&=
\lim_{k\to\infty}\frac{\#W_k}{\#(S\cap I_k)}=\overline{d}(A|S)>0.\end{split}\end{equation}
 This implies that
$W_k$ is non-empty for all sufficiently large $k$.
Since 
$\max I_k<\min I_{k+1}$ for every $k\geq1$ by \eqref{up-eq2*}, 
$\{W_k\colon k\in\mathbb N\}$ consists of pairwise disjoint, finite subsets of $A\setminus S_*$.
We set
\begin{equation}\label{target-def}E_{S,A}=R_{t,L}(S_*,A\setminus S_*,(M_k)_{k=1}^\infty,(W_k)_{k=1}^\infty).\end{equation}
The construction of this set was described in Section~\ref{insert-sec}.
For each $y\in R_{t,L}(S_*)$ 
we add the elements of $W_{k}$ into the digit sequence $(a_n(y))_{n=1}^\infty$ at position $M_k$ to define a new sequence. Then $E_{S,A}$ is the set of points whose regular continued fraction expansions are given by these new sequences.
It is clear from the construction that for any $x\in E_{S,A}$, the digit sequence $\{a_n(x)\}_{n\in\mathbb N}$ is not monotone increasing, as shown in \textsc{Figure}~\ref{fig1}.

We claim that $ E_{S,A}\subset E$. 
Indeed, let $x\in  E_{S,A}$ and let $y$ be the corresponding point in $R_{t,L}(S_*)$, namely 
$y=f(x)$.
Since $(a_n(y))_{n=1}^\infty$ is a strictly increasing sequence in $S_*$ and $\bigsqcup_{k=1}^\infty W_{k}\subset A\setminus S_*$, for all $m$, $n\in\mathbb N$ with $m\neq n$ we have $a_m(x)\neq a_n(x)$, namely $x\in E$.
By the construction of $E_{S,A}$, we have \[\#W_k\leq\#\left(\bigcup_{n\in\mathbb N}\bigcap_{x\in E_{S,A}}\{a_n(x)\}\cap A\cap I_k\right)\leq \#(A\cap I_k),\] and \eqref{reladens-PS*} implies
\begin{equation}\label{final}\lim_{k\to\infty}\frac{\#\left(\bigcup_{n\in\mathbb N}\bigcap_{x\in E_{S,A}}\{a_n(x)\}\cap A\cap I_k\right) }{\# (S\cap I_k)}=\overline{d}(A|S).\end{equation}

Since \eqref{holder-gen-eq0} and
\eqref{holder-ineq0} 
follow from \eqref{up-eq3*} and \eqref{up-eq4*} respectively,
the corresponding elimination map 
is almost Lipschitz by Proposition~\ref{holder-ex-lem1}. 
Lemma~\ref{Holder} gives
\begin{equation}\label{lower-dim-b}\dim_{\rm H}E_{S,A}\ge \dim_{\rm H}R_{t,L}(S_*)=\frac{1}{2\alpha}.\end{equation}
It is clear from the construction that $E_{S,A}\subset\{x\in E\colon\{a_n(x)\colon n\in\mathbb N\}\subset S\}$. 
Combining \eqref{lower-dim-b} and
Proposition~\ref{E-cor} we obtain \eqref{equation4}.

Since $S$ has polynomial density with exponent $\alpha\geq1$, there exist $C_1,C_2>0$ and $\beta\geq0$ such that for all sufficiently large $k\geq1$ we have 
\[\frac{\#(S\cap [1,\min I_k))}{\#(S\cap [1,\max I_k])}\leq C_1^{-1}C_2\frac{(\min I_k)^{1/\alpha}}{(\max I_k)^{1/\alpha}}\frac{(\log(\max I_k))^{\beta}}{(\log(\min I_k))^\beta}.\]
On the right-hand side, 
using $\lim_{k\to\infty}\min I_k=\infty$ and \eqref{I-limit} we have \[\lim_{k\to\infty}\frac{(\min I_k)^{1/\alpha}}{(\max I_k)^{1/\alpha}}\frac{(\log(\max I_k))^\beta}{(\log(\min I_k))^\beta}=\lim_{k\to\infty}\left(\frac{\min I_k}{\max I_k}\right)^{1/\alpha}\left(1+\frac{\log\frac{\max I_k}{\min I_k}}{\log\min I_k}\right)^\beta=0,\]
and hence \[\lim_{k\to\infty}\frac{\#(S\cap I_k)}{\#(S\cap [1,\max I_k])}=1-\lim_{k\to\infty}\frac{\#(S\cap [1,\min I_k)) }{\#(S\cap [1,\max I_k]) }=1.\]
Using this convergence and \eqref{final} 
 we obtain
\[\begin{split}\overline{d}\left(\bigcup_{n\in\mathbb N}\bigcap_{x\in E_{S,A}}\{a_n(x)\}\cap A|S\right)\geq&\limsup_{k\to\infty}\frac{\#\left(\bigcup_{n\in\mathbb N}\bigcap_{x\in E_{S,A}}\{a_n(x)\}\cap A\cap[1,\max I_k]\right) }{\#(S\cap[1,\max I_k])}\\
\geq&\limsup_{k\to\infty}\frac{\#\left(\bigcup_{n\in\mathbb N}\bigcap_{x\in E_{S,A}}\{a_n(x)\}\cap A\cap I_k\right) }{\#(S\cap[1,\max I_k])}\\
=&\lim_{k\to\infty}\frac{\#(S\cap I_k)}{\#(S\cap [1,\max I_k])}\\
&\times\lim_{k\to\infty}\frac{\#\left(\bigcup_{n\in\mathbb N}\bigcap_{x\in E_{S,A}}\{a_n(x)\}\cap A\cap I_k  \right)}{\# (S\cap I_k)}\\
=&\overline{d}(A|S).\end{split}\]
The reverse inequality $\overline{d}(\bigcup_{n\in\mathbb N}\bigcap_{x\in E_{S,A}}\{a_n(x)\}\cap A|S) \leq\overline{d}(A|S)$ is obvious.
We have verified \eqref{equation5}. Combining \eqref{equation4} and \eqref{equation5} yields \eqref{equation6}. The proof of Theorem~\ref{cor-FS-R} is complete.\qed

 \subsection{Proof of Theorem~\ref{cor-FS}(b) }\label{pf-sec}
 In principle we proceed along the line of Section~\ref{gen-sec}, but there are a number of technical differences.
 Let $S\subset\mathbb N$ satisfy
 $\overline{d}(S)=0$ and $d_{\rm B}(S)>0$. We set $M_0=N_0=0$ for convenience, and
 take two sequences $(M_k)_{k=1}^\infty$, $(N_k)_{k=1}^\infty$ of positive integers such that $(N_k)_{k=1}^\infty$ is  strictly increasing and  
 \[\lim_{k\to\infty}\frac{\#({S\cap[M_k,M_k+N_k]})}{N_k}=d_{\rm B}(S).\]
If
  $\sup_{k\geq1}M_k<\infty$ 
then we have $d_{\rm B}(S)=\overline{d}(S)$. This case is covered by Theorem~\ref{cor-FS}(a).
 In what follows we assume 
  $\sup_{k\geq1}M_k=\infty$.
 Taking subsequences of 
$(M_k)_{k=1}^\infty$ and 
$(N_k)_{k=1}^\infty$ if necessary, 
we may assume $(M_k)_{k=1}^\infty$ is strictly increasing and $M_k+N_k<M_{k+1}$ for every $k\geq1$. For each $k\in\mathbb N$ we put
\[I_{k}=[M_{k-1},M_{k-1}+N_{k-1}]\]
and
\[W_k=S\cap I_k.\]
We have $\lim_{k\to\infty}\#(\mathbb N\cap I_k)=\infty$ and 
\begin{equation}\label{ban-eq1} \lim_{k\to\infty}\frac{\#W_k}{\#(\mathbb N\cap I_k)}= d_{\rm B}(S).\end{equation}

Let $(\varepsilon_k)_{k=1}^\infty$ be a decreasing sequence of positive reals converging to $0$.
Taking subsequences of
$(M_k)_{k=1}^\infty$ and $(N_k)_{k=1}^\infty$ if necessary, we may assume 
 \begin{equation}\label{eq-100}\frac{1}{(M_{k-1}+N_{k-1}+1)^{2 N_{k-1} }}
 \geq\left(\prod_{i=M_{k-1}+1}^{M_{k}}\frac{1}{t^{2i}} \right)^{3\varepsilon_k}\ \text{ for every }k\geq1,\end{equation}
and
\begin{equation}\label{holMk}
\inf_{n\geq M_k}\prod_{i=1}^{n+1}\frac{1}{(t^{i+2})^2}\cdot\left(\prod_{i=1}^{n }\frac{1}{t^{2i}}\right)^{-1-\varepsilon_k}\geq2\ \text{ for every }k\geq1.\end{equation}
 Then 
\eqref{eq-100} implies
 \begin{equation}\label{eq-300}\prod_{i\in W_{k} }\frac{1}{(i+1)^{2}}\geq\left(\prod_{i=M_{k-1}+1}^{M_{k}}\frac{1}{t^{2i}} \right)^{3\varepsilon_k}\ \text{ for every }k\geq1\text{ with }S\cap I_k\neq\emptyset.\end{equation}

Since $\overline{d}(S)=0$, the set $\mathbb N\setminus S$ has polynomial density
with $\alpha=1$, $\beta=0$.
By Proposition~\ref{seed-Prop}, there exist a subset $(\mathbb N\setminus S)_*$ of $\mathbb N\setminus S$ not containing consecutive integers, and positive reals $t,L$ with $t>L>1$ and $t>2$ such that 
$R_{t,L}((\mathbb N\setminus S)_*)$ is an extreme seed set of Hausdorff dimension $1/2$. 
We set
\begin{equation}\label{target-def-F}F_S=R_{t,L}((\mathbb N\setminus S)_*,S,(M_k)_{k=1}^\infty,
(W_k)_{k=1}^\infty).\end{equation}
The construction of this set was described in Section~\ref{insert-sec}.
For each $y\in R_{t,L}((\mathbb N\setminus S)_*)$
we add the elements of $W_{k}$ into the digit sequence $(a_n(y))_{n=1}^\infty$ at position $M_k$ to define a new sequence. Then $F_{S}$ is the set of points whose regular continued fraction expansions are given by these new sequences.
It is clear from the construction that for any $x\in F_{S}$, the digit sequence $\{a_n(x)\}_{n\in\mathbb N}$ is not monotone increasing.

Since the intervals $I_k$, $k\geq1$ are pairwise disjoint and $(\mathbb N\setminus S)_*\cap S=\emptyset$,   we have 
$F_S\subset E$.
Since 
\eqref{holder-gen-eq0} and \eqref{holder-ineq0} follow from  \eqref{eq-300} and 
\eqref{holMk} respectively,
 the corresponding elimination map is almost Lipschitz by Proposition~\ref{holder-ex-lem1}. 
  Lemma~\ref{Holder} gives
\[\dim_{\rm H}F_S\ge \dim_{\rm H}R_{t,L}((\mathbb N\setminus S)_*)= \frac{1}{2}.\]

By the construction of $F_S$, we have \[\bigcup_{n\in\mathbb N}\bigcap_{x\in F_S}\{a_n(x)\}\cap W_k=W_k.\] 
Then \eqref{ban-eq1} yields
\[\begin{split}d_{\rm B}\left(\bigcup_{n\in\mathbb N}\bigcap_{x\in F_S}\{a_n(x)\}\cap S\right)&\geq\limsup_{k\to\infty}\frac{\#
\left(\bigcup_{n\in\mathbb N}\bigcap_{x\in F_S}\{a_n(x)\}\cap W_k\right) }{\# (\mathbb N\cap I_k)}\\
&=\lim_{k\to\infty}\frac{\#W_k}{\#(\mathbb N\cap I_k)}= d_{\rm B}(S).\end{split}\]
The reverse inequality 
$d_{\rm B}(\bigcup_{n\in\mathbb N}\bigcap_{x\in F_S}\{a_n(x)\}\cap S)\leq d_{\rm B}(S)$ is obvious. 
 We have verified \eqref{equation1'}, which implies \eqref{equation-1}.
 The proof of Theorem~\ref{cor-FS}(b) is complete.
\qed

\subsection{Proof of Theorem~\ref{J-thm}}\label{pfthmj}
Let $S$ be an infinite subset of $\mathbb N$.
By Lemma~\ref{disc2}, if $x\in E$ then
$\overline{d}(\{a_n(x)\colon n\in\mathbb N\}\cap S)>0$ implies
\[\tau(\{a_n(x)\colon n\in\mathbb N\}\cap S)= 1= \tau(\{a_n(x)\colon n\in\mathbb N\}).\]
Applying
Lemma~\ref{fang-eq} with $c=1$, we conclude that 
\[\dim_{\rm H}
\left\{x\in J\colon 
\ \overline{d}(\{a_n(x)\colon n\in\mathbb N\}\cap S)>0
\right\}=0.\] 
Since $\dim_{\rm H}J=1/2>0$ by \cite{R85} or 
 Proposition~\ref{E-cor},
the desired inequality in Theorem~\ref{J-thm}(a) holds.

Suppose $S$ has polynomial density with exponent $\alpha<2$. Lemma~\ref{disc2} gives $\tau(S)=1/\alpha$.
Let $A\subset S$ satisfy $\overline{d}(A|S)>0$. By Lemma~\ref{disc2},
 if $x\in E$ then
$\overline{d}(\{a_n(x)\colon n\in\mathbb N\}\cap A|S)>0$ implies
\[\tau(\{a_n(x)\colon n\in\mathbb N\}\cap A)= \frac{1}{\alpha}\leq\tau(\{a_n(x)\colon n\in\mathbb N\}).\]
If moreover $\{a_n(x)\colon n\in\mathbb N\}\subset S$, then we have
\[\tau(\{a_n(x)\colon n\in\mathbb N\})\leq\tau(S)=\frac{1}{\alpha}.\]
Further, applying
Lemma~\ref{fang-eq} with $c=1/\alpha$ we obtain
\[\begin{split}\dim_{\rm H}\left\{
\begin{tabular}{l}
\!\!\!$x\in J\colon$\\
\!\!\!$\{a_n(x)\colon n\in\mathbb N\}\subset S$,\!\!\!\\
\!\!\!$\overline{d}(\{a_n(x)\colon n\in\mathbb N\}\cap A |S)>0$\!\!\!
\end{tabular}
\right\}&\leq
\dim_{\rm H}\left\{x\in J\colon\tau(\{a_n(x)\colon n\in\mathbb N\})=\frac{1}{\alpha}\right\}\\
&\leq\frac{\alpha-1}{2\alpha}.\end{split}\]
Combining this inequality, the equality in Proposition~\ref{E-cor} and then using $\alpha<2$ yields the desired inequality in Theorem~\ref{J-thm}(b).\qed

\section{Some Generalizations}
In this last section, we generalize Theorems~\ref{cor-FS} and \ref{cor-FS-R}
to other expansions of real numbers. Proofs are essentially the same as that of Theorems~\ref{cor-FS} and \ref{cor-FS-R}, and hence will be omitted.

\subsection{Iterated Function Systems}
A countable collection of self-maps 
is called an {\it Iterated Function System (IFS)} \cite{Fal14}.  
Let $\Phi=\{\varphi_k\colon k\in\mathbb N\}$ be a collection of $C^1$ maps from $[0,1]$ to itself and let $d>1$. 
We say $\Phi$ is a {\it $d$-decaying  
IFS} \cite{JR12} if the following conditions hold:
\begin{itemize}
\item[(i)]there exists $m\in\mathbb N$ and  $\gamma\in (0, 1)$ such that for every $(a_1,\ldots, a_m)\in \mathbb N^{m}$ and every $x\in [0, 1]$ we have 
$0<|(\varphi_{a_1}\circ\cdots\circ \varphi_{a_m})^{\prime}(x)|\le \gamma;$
\item[(ii)]for all $i,j\in \mathbb N$ with $i\neq j$, $\varphi_{i}((0, 1))\cap \varphi_{j}((0,1))=\emptyset;$
\item[(iii)] for any $\varepsilon>0,$ there exist positive constants $c_1(\varepsilon)$, $c_2(\varepsilon)$ such that for every $n\in \mathbb N$ and every $x\in [0, 1]$ we have 
\[\frac{c_1(\varepsilon)}{k^{d(1+\varepsilon)}}\le |\varphi_k^{\prime}(x)|\le \frac{c_2(\varepsilon)}{k^{d(1-\varepsilon)}};\]
\item[(iv)]$\overline{\bigcup_{k=1}^{\infty}\varphi_k([0, 1])}=[0, 1]$, and if $i< j$ then $\varphi_i(x)>\varphi_j(x)$ for all $x\in [0, 1].$
\end{itemize}

If $\Phi=\{\varphi_k\colon k\in\mathbb N\}$  is a $d$-decaying IFS, then 
by (i) one can define a natural projection $\Pi\colon \mathbb N^{\mathbb N}\rightarrow [0, 1]$ by 
\[\Pi((a_n)_{n=1}^{\infty})=\lim_{n\to \infty}\varphi_{a_1}\circ\cdots\circ \varphi_{a_n}(1).\]
By (ii), $\Pi$ is one-to-one
 except on a countable number of points.
If $x\in [0, 1]$ and $\Pi^{-1}(x)$ is a singleton, then there exists a unique sequence $(a_n(x))_{n=1}^{\infty}$ of positive integers such that \[x=\Pi((a_n(x))_{n=1}^{\infty}).\]
This equation may be viewed as an expansion of $x$ with 
the digit sequence $(a_n(x))_{n=1}^{\infty}$.
As an analogue of the set $E$ for the regular continued fraction,
define 
$E(\Phi)$ to be the set of $x\in [0,1]$ such that $\Pi^{-1}(x)$ is a singleton and $a_{m}(x)\neq a_{n }(x)$ holds for all $m$, $n\in\mathbb N$ with $m\neq n$.
Theorems~\ref{cor-FS} and \ref{cor-FS-R} can be generalized to  expansions generated by arbitrary $d$-decaying IFSs as follows.

\begin{thm}\label{thmC1}
Let $d>1$ and let $\Phi$ be a $d$-decaying IFS. Let
$S\subset\mathbb N$. 

\begin{itemize}
\item[(a)]
If 
$d_{\rm B}(S)>0$, then
there exists $F_{S}(\Phi)\subset E(\Phi)$ 
such that \[ \dim_{\rm H}F_{S}(\Phi)=\dim_{\rm H}E(\Phi)\
\text{ and }\
d_{\rm B}\left(\bigcup_{n\in\mathbb N}\bigcap_{x\in F_S(\Phi)}\{a_n(x)\}\cap S\right)=d_{\rm B}(S).\]
In particular, 
\[ \dim_{\rm H}\left\{x\in E(\Phi)\colon d_{\rm B}(\{a_n(x)\colon n\in\mathbb N\}\cap S)=d_{\rm B}(S)\right\}=\dim_{\rm H}E(\Phi).\]

\item[(b)] If $S$ has polynomial density with exponent $\alpha\geq1$, then
for any $A\subset S$ with $\overline{d}(A|S)>0$ 
there exists a subset $E_{S,A}(\Phi)$ of $\{x\in E(\Phi) \colon \{a_n(x)\colon n\in\mathbb N\}\subset  S \}$ such that \[\dim_{\rm H}E_{S,A}(\Phi)=\dim_{\rm H}\{x\in E(\Phi)\colon \{a_n(x)\colon n\in\mathbb N\}\subset S\}=\frac{1}{d\alpha}\] and
\[\overline{d}\left(\bigcup_{n\in\mathbb N}\bigcap_{x\in E_{S,A}(\Phi)}\{a_n(x)\}\cap A|S\right)=\overline{d}(A|S).\]
In particular, 
\[\dim_{\rm H}\left\{
\begin{tabular}{l}
\!\!\!$x\in E(\Phi)\colon$\\
\!\!\!$\{a_n(x)\colon n\in\mathbb N\}\subset S$,\!\!\!\\
\!\!\!$\overline{d}(\{a_n(x)\colon n\in\mathbb N\}\cap A |S)$\!\!\!\\
\!\!\!$=\overline{d}(A|S)$\!\!\!
\end{tabular}
\right\}=\dim_{\rm H}
\left\{\begin{tabular}{l}
\!\!\!$x\in E(\Phi)\colon$\\
\!\!\!$\{a_n(x)\colon n\in\mathbb N\}\subset S$\!\!\!\end{tabular}
\right\}.\]
\end{itemize}
\end{thm}

The regular continued fraction is generated by the $2$-decaying IFS $\{\varphi_k\colon k\in\mathbb N\}$ defined by 
 $\varphi_k(x)= 1/(x+k)$ with $m=2$. Hence, Theorem~\ref{thmC1} is indeed a generalization of Theorems~\ref{cor-FS} and \ref{cor-FS-R}.
 Then, Theorems~\ref{cor-N}, \ref{TZ-thm}, \ref{cor-prime}, \ref{cor-P1}, \ref{cor-PS} and \ref{lower-cor} for the regular continued fraction are generalized to $d$-decaying IFSs accordingly.
Precise statements are omitted.

Regarding a proof of Theorem~\ref{thmC1},
 most of the arguments in the proofs of Theorems~\ref{cor-FS} and \ref{cor-FS-R} remain intact. 
Lemma~\ref{proper} is the only place where the specific formulas for the regular continued fraction were used.
For each $n\in\mathbb N$, the intervals $\varphi_{a_1}\circ\cdots\circ\varphi_{a_n}([0,1])$,
$(a_1,\ldots, a_n)\in\mathbb N^n$ play the role of the $n$-th fundamental intervals.
As a substitute for Lemma~\ref{proper}, one can use
(iii) to evaluate the diameters of these intervals.

\subsection{Semi-regular continued fractions}
Let
$\sigma=\{\sigma_n\}_{n=1}^\infty\in\{-1,1\}^\mathbb N$. Any 
number $x$ in the interval $(0,1)$ has the unique, finite or infinite continued fraction expansion of the form 
\begin{equation}\label{CF-general}
x=
\confrac{1 }{a_{\sigma,1}(x)} + \confrac{\sigma_1 }{a_{\sigma,2}(x)}  + \confrac{\sigma_{2} }{a_{\sigma,3}(x) }+\cdots,
\end{equation}
where 
 $a_{\sigma,n}(x)$ are positive integers such that $\sigma_n+a_{\sigma,n}(x)\geq1$
for 
every $n\geq1$ \cite{KKV17,NT}. This kind of continued fractions are called {\it semi-regular} \cite{IK,Kra91}. 
If $x$ is irrational then the continued fraction in \eqref{CF-general}
is infinite, which means that
the finite truncation
\[\cfrac{1}{a_{\sigma,1}(x)+\cfrac{\sigma_1 }{a_{\sigma,2}(x)+\cdots+\cfrac{\ddots}{\displaystyle{\frac{\sigma_{n-1}}{a_{\sigma,n}(x)}}}}}\]
converges to $x$ as $n\to\infty$.
If $\sigma_n=1$ for all $n\in\mathbb N$ (resp. $\sigma_n=-1$ for all $n\in\mathbb N$) then we get the regular (resp. backward) continued fraction expansion. 
For other examples and motivations for semi-regular continued fractions, see
\cite{IK,Kra91} and the references therein.

As an analogue of the set $E$ for the regular continued fraction,
for each $\sigma\in\{-1,1\}^{\mathbb N}$
let
$E_\sigma$ denote the set of $x\in (0,1)\setminus\mathbb Q$ such that  $a_{\sigma, m}(x)\neq a_{\sigma, n }(x)$ holds for all $m$ ,$n\in\mathbb N$ with $m\neq n$.
Theorems~\ref{cor-FS} and \ref{cor-FS-R} can be generalized to semi-regular continued fractions as follows.
\begin{thm}\label{thmC2}
Let $S\subset\mathbb N$.

\begin{itemize}
\item[(a)] If 
$d_{\rm B}(S)>0$, then for any $\sigma\in\{-1,1\}^\mathbb N$
there exists $F_{S,\sigma}\subset E_\sigma$ 
such that \[\dim_{\rm H}F_{S,\sigma}=\dim_{\rm H}E_\sigma\
\text{ and }\
d_{\rm B}\left(\bigcup_{n\in\mathbb N}\bigcap_{x\in F_{S,\sigma}}\{a_{n,\sigma}(x)\}\cap S\right)=d_{\rm B}(S).\]
In particular, 
\[\dim_{\rm H}\left\{x\in E_\sigma\colon d_{\rm B}(\{a_{n,\sigma}(x)\colon n\in\mathbb N\}\cap S)=d_{\rm B}(S)\right\}=\dim_{\rm H}E_\sigma.\]

\item[(b)] If $S$ has polynomial density with exponent $\alpha\geq1$, then
for any $A\subset S$ with $\overline{d}(A|S)>0$ and any $\sigma\in\{-1,1\}^\mathbb N$, 
there exists a subset $E_{S,A,\sigma}$ of $\{x\in E_\sigma \colon \{a_{\sigma,n}(x)\colon n\in\mathbb N\}\subset  S \}$ such that \[\dim_{\rm H}E_{S,A, \sigma}=\dim_{\rm H}\{x\in E_\sigma\colon \{a_{\sigma,n}(x)\colon n\in\mathbb N\}\subset S\}=\frac{1}{2\alpha}\] and
\[\overline{d}\left(\bigcup_{n\in\mathbb N}\bigcap_{x\in E_{S,A, \sigma}}\{a_{\sigma,n}(x)\}\cap A|S\right)=\overline{d}(A|S).\]
In particular, 
\[\dim_{\rm H}\left\{
\begin{tabular}{l}
\!\!\!$x\in E_\sigma\colon$\\
\!\!\!$\{a_{\sigma,n}(x)\colon n\in\mathbb N\}\subset S$,\!\!\!\\
\!\!\!$\overline{d}(\{a_{\sigma,n}(x)\colon n\in\mathbb N\}\cap A |S)$\!\!\!\\
\!\!\!$
=\overline{d}(A|S)$\!\!\!
\end{tabular}
\right\}=\dim_{\rm H}
\left\{\begin{tabular}{l}
\!\!\!$x\in E_\sigma\colon$\\
\!\!\!$\{a_{\sigma,n}(x)\colon n\in\mathbb N\}\subset S$\!\!\!\end{tabular}
\right\}.\]\end{itemize}
\end{thm}

By virtue of Theorem~\ref{thmC2}, Theorems~\ref{cor-N}, \ref{TZ-thm}, \ref{cor-prime}, \ref{cor-P1}, \ref{cor-PS} and \ref{lower-cor} are generalized to semi-regular continued fractions accordingly. Precise statements are omitted.

Regarding a proof of Theorem~\ref{thmC2},
 most of the arguments in the proofs of Theorems~\ref{cor-FS} and \ref{cor-FS-R} remain intact apart from Lemma~\ref{proper}.
 For each $\sigma\in\{-1,1\}^\mathbb N$,
$n\in\mathbb N$ and $(a_1,\ldots,a_n)\in\mathbb N$ such that $\sigma_i+a_{i}\geq1$
for $i=1,\ldots,n$, we define
\[I_\sigma(a_1,\ldots, a_n)=\{x\in (0, 1)\colon a_{\sigma,i}(x)=a_i\ \text{ for }i=1,\ldots,n\}.\]
These sets are non-degenerate intervals, and play the role of the $n$-th fundamental intervals.
The following can be used as a substitute for Lemma~\ref{proper}.

\begin{lem}\label{lem-last}
There exists a constant $C\geq1$ such that
if $\sigma\in\{-1,1\}^\mathbb N$, $n\in\mathbb N$, $(a_1,\ldots,a_n)\in\mathbb N^n$ are such that $\sigma_i+a_{i}\geq1$
and $a_i\geq3$ for $i=1,\ldots,n$ then 
\[C^{-1}\prod_{i=1}^n\frac{1}{(a_i+1)^2}\leq|I_\sigma(a_1,\ldots,a_n)|\leq C\prod_{i=1}^n\frac{1}{(a_i-1)^2}.\]\end{lem}
 
 Similarly to the case of the regular continued fraction, the endpoints of the intervals $I_\sigma(a_1,\ldots,a_n)$ can be written with the partial numerators and denominators of the semi-regular continued fraction expansion
(see \cite{Kra91}). 
Hence one can prove Lemma~\ref{lem-last} similarly to the proof of Lemma~\ref{proper}. Here we give a conceptually clearer,  
  dynamical alternative proof.

   \begin{proof}[Proof of Lemma~\ref{lem-last}]
Define two maps $T_1\colon x\in(0,1]\mapsto 1/x-\lfloor 1/x\rfloor$ and $T_{-1}\colon x\in(0,1]\mapsto \lfloor 1/x\rfloor-1/x+1$.
For each $k\in\mathbb N$, $T_1$ and $T_{-1}$ send the interval
$(\frac{1}{k+1},\frac{1}{k})$ diffeomorphically onto $(0,1)$.
They satisfy R\'enyi's condition
\[
 \sup_{\left(\frac{1}{k+1},\frac{1}{k}\right)}\frac{|T_1''|}{|T_1'|^2}\leq2\quad\text{for all }k\in\mathbb N \ \text{ and }\sup_{\left(\frac{1}{k+1},\frac{1}{k}\right)}\frac{|T_{-1}''|}{|T_{-1}'|^2}\leq2\quad\text{for all }k\in\mathbb N.
 \]

Write $T_\sigma^n=T_{\sigma_n}\circ \cdots\circ T_{\sigma_1}$,  and for convenience set $T_\sigma^0$ to be the identity map on $(0,1)$. Let $\theta$ denote the left shift acting on $\{-1,1\}^{\mathbb N}$. We have
$T_\sigma^i(I_\sigma(a_1,\ldots,a_n))=I_{\theta^i\sigma}(a_{i+1},\ldots,a_n)$
for $i=1,\ldots,n-1$, and
$T_\sigma^n$ sends $I_\sigma(a_1,\ldots,a_n)$ diffeomorphically onto $(0,1)$. 
For all $x\in I_\sigma(a_1,\ldots,a_n)$ and $i=1,\ldots,n$ we have
\begin{equation}\label{derivative}(a_i-1)^2\leq |(T_{\sigma_i})'T_\sigma^{i-1} x|\leq(a_i+1)^2.\end{equation} 
Since $a_1,\ldots,a_n\geq3$, combining the first inequality in \eqref{derivative} and
 R\'enyi's condition we obtain a constant
 $C\geq1$ independent of $\sigma$, $n$ such that 
\[\sup_{x,y\in I_{\sigma}(a_1,\ldots,a_n)}\frac{|(T_\sigma^n)'y|}{|(T_\sigma^n)'x|}\leq C.\]
By the mean value theorem, there exists $x\in I_{\sigma}(a_1,\ldots,a_n)$ such that $C^{-1}|(T_\sigma^n)'x|^{-1}\leq|I_\sigma(a_1,\ldots,a_n)|\leq C|(T_\sigma^n)'x|^{-1}$.
The chain rule gives
$|(T^n_\sigma)'x|=\prod_{i=1}^n|(T_{\sigma_i})'T_\sigma^{i-1} x|$.
Using \eqref{derivative} we obtain the desired double inequalities.
\end{proof}

\appendix \def\thesection{\Alph{section}}
\section{Bergelson-Leibman's theorem in $\mathbb N$}\label{BL-appendix}
This appendix gives the proof of Theorem~\ref{BL-thm}. It relies on Furstenberg's correspondence principle and the following multiple recurrence theorem with polynomial times.
Recall that $\mathbb Z_0[X]$ denotes the set of polynomials with integer coefficients that vanish at $X=0$.
 \begin{thm}[{\cite[Theorem~${\rm A}_0$]{BL}}]\label{BL-thmA} Let $(X,\mathcal B,\mu)$ be a probability space, let $T_1,\ldots,T_\ell$ be a finite collection of commuting invertible measure preserving transformations on $X$, let $h_1,\ldots,h_\ell\in\mathbb Z_0[X]$, and let $A\in\mathcal B$ satisfy $\mu(A)>0$. Then
 \[\liminf_{N\to\infty}\frac{1}{N}
 \sum_{m=0}^{N-1}\mu(T_1^{-h_1(m)}(A)\cap T_2^{-h_2(m)}(A)\cap\cdots\cap T_\ell^{-h_\ell(m)}(A))>0.\]
\end{thm}
\subsection{Proof of Theorem~\ref{BL-thm}}Let $S\subset\mathbb N$ satisfy $d_{\rm B}(S)>0$.
Let $\{0,1\}^\mathbb Z$ denote the two-sided Cartesian product space endowed with the product topology of the discrete topology on $\{0,1\}$.
Define $x^*=(x_i^*)_{i=-\infty}^\infty\in\{0,1\}^\mathbb Z$ by
\[x_i^*=\begin{cases}1 &\text{if $i\geq1$ and $i\in S$,}\\
0 &\text{otherwise.}\end{cases}\]
Let $\theta$ denote the left shift acting on $\{0,1\}^\mathbb Z$ and let
$\Sigma=\overline{\{\theta^i(x^*)\colon i\in\mathbb Z\}}$. Then $\Sigma$ is a compact metrizable space and  $\theta(\Sigma)=\Sigma$. Let
$A=\{(y_i)_{i=-\infty}^{\infty}\in \Sigma\colon y_0=1\}$. 
Let $(I_k)_{k=1}^\infty$ be a sequence of compact intervals in $[1,\infty)$ such that $\lim_{k\to\infty}\# (\mathbb N\cap I_k)=\infty$ and
\[d_{\rm B}(S)=\lim_{k\to\infty}\frac{\#(S\cap I_k)}{\# (\mathbb N\cap I_k)}.\]
For each $k\geq1$, define a Borel probability measure $\mu_k$ on $\Sigma$ by
\[\mu_k=\frac{1}{\#(\mathbb N\cap I_k)}\sum_{i\in \mathbb N\cap I_k}\delta_{\theta^i(x^*)},\]  where $\delta_{\theta^i(x^*)}$ denotes the unit point mass at $\theta^i(x^*)$. 
Note that
$\lim_{k\to\infty}\mu_k(A)=d_{\rm B}(S)$.
Pick a limit point of the sequence $(\mu_k)_{k=1}^\infty$ in the weak* topology and denote it by $\mu$. Then
 $\mu$ is $\theta|_\Sigma$-invariant.
Since $A$ is a clopen set, we have
$\mu(A)=\lim_{k\to\infty}\mu_k(A)$ and so $\mu(A)=d_{\rm B}(S)>0$.

For each $m\in\mathbb Z$ define \[S-m=\{i\in\mathbb N\colon i+m\in S\}.\]
Let $k\geq1$ and $m_1,\ldots,m_\ell\in\mathbb Z$. 
Note that
$i\in I_k\cap S\cap(S-m_1)\cap\cdots\cap(S-m_\ell)$ if and only if
$\theta^i(x^*)\in A\cap \theta^{-m_1}(A)\cap\cdots\cap\theta^{-m_\ell}(A)$. This correspondence yields
\[\begin{split}\frac{\#(I_k\cap S\cap(S-m_1)\cap\cdots\cap(S-m_\ell))}{\#(\mathbb N\cap I_{k})}=\mu_k(A\cap \theta^{-m_1}(A)\cap\cdots\cap \theta^{-m_\ell}(A)).\end{split}\]
Hence
\[\begin{split}d_{\rm B}(S\cap(S-m_1)\cap\cdots\cap(S-m_\ell ))&\geq\lim_{k\to\infty}\frac{\#(I_k\cap S\cap(S-m_1)\cap\cdots\cap(S-m_\ell))}{\#(\mathbb N\cap I_{k})}\\
&=\lim_{k\to\infty}\mu_k(A\cap \theta^{-m_1}(A)\cap\cdots\cap \theta^{-m_\ell}(A))\\
&=\mu(A\cap \theta^{-m_1}(A)\cap\cdots\cap \theta^{-m_\ell}(A)).\end{split}\]
The first inequality follows from the definition of upper Banach density. The second equality holds because $A\cap \theta^{-m_1}(A)\cap\cdots\cap \theta^{-m_\ell}(A)$ is a clopen set.
Let $h_1,\ldots,h_\ell\in\mathbb Z_0[X]$. Replacing $m_j$ by $h_j(m)$, $m\in\mathbb N$ for $j=1,\ldots,\ell$ we have
\[d_{\rm B}(S\cap(S-h_1(m))\cap\cdots\cap(S-h_\ell(m) ))\geq\mu(A\cap \theta^{-h_1(m)}(A)\cap\cdots\cap \theta^{-h_\ell(m)}(A)).\]
Applying Theorem~\ref{BL-thmA} to the invertible $\mu$-preserving system $(\Sigma,\theta|_\Sigma)$, we obtain
\[d_{\rm B}(S\cap(S-h_1(m))\cap\cdots\cap(S-h_\ell(m) ))>0\] for infinitely many $m\in\mathbb N$.
This yields the desired statement.\qed

\section{Arithmetic progressions in ${\rm PS}(\alpha)$ }\label{pfthmSY}
This appendix gives the proof of Theorem~\ref{thm-SY} for the sets obtained by the Piatetski-Shapiro sequences. Recall that ${\rm PS}(\alpha)=\{\lfloor n^\alpha\rfloor\colon n\in\mathbb N\}$ for $\alpha>1$. 
\subsection{Proof of Theorem~\ref{thm-SY}}Let $1<\alpha<2$.
Let $A\subset {\rm PS}(\alpha)$ satisfy $\overline{d}(A|{\rm PS}(\alpha))>0$. There exists a strictly increasing sequence $(N_k)_{k=1}^\infty$ of positive integers such that
for every $k\geq1$,
\[\frac{\#(A\cap[1,N_k])}{ \#({\rm PS}(\alpha)\cap[1,N_k])}\geq\frac{\overline{d}(A|{\rm PS}(\alpha))}{2}.\]
Set $B=\{n\in\mathbb N\colon \lfloor n^\alpha\rfloor\in A\}$.
For every $k\geq1$ we have
\[\frac{\#(B\cap[1,\#({\rm PS}(\alpha)\cap[1,N_k]) ])}{\#({\rm PS}(\alpha)\cap[1,N_k]) }=\frac{\#(A\cap[1,N_k])}{ \#({\rm PS}(\alpha)\cap[1,N_k]) }\geq \frac{\overline{d}(A|{\rm PS}(\alpha))}{2}>0.\]  This implies
$d_{\rm B}(B)\geq\overline{d}(B)>0$. 
By \cite[Corollary~5]{SY19},  $\{(n,\lfloor n^\alpha\rfloor)\in\mathbb N^2\colon n\in B\}$ contains an $\ell$-term arithmetic progression in $\mathbb N^2$ for every $\ell\geq3$, that is, for every $\ell\geq3$ there exist $n_1,\ldots,n_\ell\in\mathbb N$ and $m_1,m_2\in\mathbb N$ such that $n_k=n_1+m_1(k-1)$ and $\lfloor n_k^\alpha\rfloor=\lfloor n_1^\alpha\rfloor+m_2(k-1)$
for $k=2,\ldots,\ell$.
Therefore, $A=\{\lfloor n^\alpha\rfloor\colon n\in B\}$ contains an $\ell$-term arithmetic progression for every $\ell\geq3$.\qed

\subsection*{Acknowledgments} We thank anonymous referees for their careful readings of the manuscript and giving valuable comments and suggestions. We thank Yasushi Nagai and Kota Saito for their careful readings of an earlier version of the draft. 
We thank Nima Alibabaei, Atsushi Katsuda, Shunsuke Usuki for fruitful discussions. YN was supported by the JSPS KAKENHI 25K17282, Grant-in-Aid for Early-Career Scientists. HT was supported by the JSPS KAKENHI 25K21999, Grant-in-Aid for Challenging Research (Exploratory).

\end{document}